\documentclass[12pt]{amsart}
\usepackage{mathtools,amsfonts,amsmath}
\usepackage{fullpage}
\usepackage{enumerate}

\newcommand{\B}[1]{\overline{#1}}
\newcommand{\C}{\mathbb{C}}
\newcommand{\R}{\mathbb{R}}
\newcommand{\rb}{\rangle}
\newcommand{\lb}{\langle}
\newcommand{\f}[1]{\mathfrak{#1}}
\newcommand{\bgm}[2]{\lb#1,#2\rb}
\newcommand{\vcla}{\mathcal{C}}
\newcommand{\HCFp}{\operatorname{HCF}_+}
\newcommand{\id}{\operatorname{Id}}
\newcommand{\tr}{\operatorname{tr}}
\newtheorem{thm}{Theorem}[section]
\newtheorem{defn}[thm]{Definition}
\newtheorem{prop}[thm]{Proposition}

\newtheorem{lem}[thm]{Lemma}
\newtheorem{cor}[thm]{Corollary}

\theoremstyle{remark}
\newtheorem{rmk}[thm]{Remark}

\title[$\HCFp$ on nilpotent and almost-abelian complex Lie groups]{Positive Hermitian Curvature Flow on nilpotent and almost-abelian complex Lie groups}

\author{James Stanfield}

\begin{document}
	\addtocontents{toc}{\protect\setcounter{tocdepth}{2}}
		\begin{abstract}
		We study the positive Hermitian curvature flow on the space of left-invariant metrics on complex Lie groups. We show that in the nilpotent case, the flow exists for all  positive times and subconverges in the Cheeger-Gromov sense to a soliton. We also show convergence to a soliton when the complex Lie group is almost abelian. That is, when its Lie algebra admits a (complex) co-dimension one abelian ideal. Finally, we study solitons in the almost-abelian setting. We prove uniqueness and completely classify all left-invariant, almost-abelian solitons, giving a method to construct examples in arbitrary dimensions, many of which admit co-compact lattices.
	\end{abstract}
	\maketitle

	\tableofcontents
	\section{Introduction}
	Motivated by the need to find a Ricci flow analogue in non-K\"ahler Hermitian geometry, in \cite{StreetsAndTianHCF2011}, Streets and Tian introduced a family of geometric flows generalising the K\"ahler-Ricci flow to Hermitian manifolds. These flows, called \emph{Hermitian Curvature Flows} (HCFs), evolve an initial Hermitian metric in the direction of a Ricci-type tensor of the Chern connection modified with some lower order torsion terms. Various members of this family are actively being studied (see \cite[\S 3.1.3]{ustinovskiyThesis} and references therein).
	 
	We will focus on a distinguished member of the HCF family that was first introduced by Ustinovskiy in \cite{UstPosHCF2019}. More precisely, let $(M,J,g)$ be a Hermitian manifold. Let $\nabla$ denote the Chern connection, $\Omega$ the corresponding Chern curvature tensor, and $T$ its torsion. The \emph{positive Hermitian Curvature flow} ($\HCFp$) is defined by the following evolution equation:
	\begin{equation}
	\label{eqn:HCF}
	\partial_tg_t = -\Theta(g_t), \qquad g_0 = g.
	\end{equation}
	Here $\Theta$ is the \emph{torsion-twisted Chern-Ricci tensor} defined as
	\[
	\Theta(g) := S(g) + Q(g),
	\]
	where $S(g)$ is the second Chern-Ricci tensor
	\[
	S(g)_{i\B{j}} := g^{k\B l}\Omega_{k\B l i \B j},
	\]
	and $Q(g)$ is a quadratic term in $T$ given by
	\[
	Q(g)_{i\B j} :=\frac12 g^{m\B n}g^{p \B s}T_{pm \B j} T_{\B s\B n i}.
	\]
	Other members of the HCF family come from different choices for the tensor $Q$ which are quadratic in the torsion. In this instance, $Q$ is chosen so that the flow preserves certain curvature positivity conditions (see \cite{UstPosHCF2019,Ustinovskiy2020} for details), generalising this property of the (K\"ahler)-Ricci flow, hence the name \emph{positive} HCF.
 
	In this article, we consider the case where $(M,J) = (G,J)$ is a simply-connected complex Lie group, and $g$ is a left-invariant Hermitian metric. Biholomorphism invariance of $\Theta$ then implies that a left-invariant solution to (\ref{eqn:HCF}) is determined by an ODE on the Lie algebra of $G$. Thus, short-time existence and uniqueness of said solution are guaranteed by standard ODE theory. We will only refer to this left-invariant solution from now.	
	
	Our main result concerns the limiting behaviour of the $\HCFp$ (\ref{eqn:HCF}) on nilpotent Lie groups. Long-time existence follows from \cite[Theorem 6.3]{UstHCFHom2017}. Adding to this, we describe the precise asymptotic behaviour of appropriately normalised solutions.
	\begin{thm}
	\label{thm:nilpConvThm}
	Let $(G,J,g)$ be a simply-connected, complex nilpotent Lie group with left-invariant Hermitian metric $g$. Suppose $g_t$ is the solution to (\ref{eqn:HCF}) with $g_0 = g$. Then $g_t$ exists for all positive times and the rescaled Hermitian manifolds $(G,J,(1+t_k)^{-1}g_{t_k})$ sub-converge in the Cheeger-Gromov topology as $t\to\infty$ to an $\HCFp$ soliton $(G_\infty,J_\infty,g_\infty)$.
	\end{thm}

	An \emph{$\HCFp$-soliton} is a Hermitian manifold with corresponding solution to (\ref{eqn:HCF}) that is self-similar. More precisely, the solution evolves only by scaling and pull-back by time-dependent biholomorphisms. We call a soliton shrinking, steady or expanding, if the scaling factor is decreasing, constant, or increasing respectively. In general, the limit $(G_\infty,J_\infty,g_\infty)$ will be non-flat (i.e. $\Theta(g_\infty)\neq 0$),  simply-connected and nilpotent but could be non-isomorphic to $G$. We say a sequence $\{(G,J,g_l)\}_{l=1}^\infty$ of Hermitian manifolds converges to $(\bar G, \bar J, \bar g)$ in the Cheeger-Gromov topology if there exist biholomorphisms $\phi_l \colon U_l \subset (\bar G, \bar J) \to \phi_l(U_l)\subset (G,J)$ mapping the identity of $\bar G$ to the identity of $G$ such that the family of open sets $\{U_l\}_{l=1}^\infty$ exhaust $G$ and $\phi_l^* g_{l} \to \bar g$ as $l\to \infty$ smoothly and uniformly over compact subsets.
	
	Theorem \ref{thm:nilpConvThm} was previously known for $2$-step nilpotent $G$ \cite[Theorem A]{pujia2020positive}. In fact this forms the base case in our proof which is inductive. More precisely, suppose $g_t$ evolves under the $\HCFp$ on the complex $k$-step nilpotent Lie group $G$, then if $Z$ is the centre of $G$, it follows from \cite[Theorem 5.1]{UstHCFHom2017} that the induced, left-invariant Hermitian metrics on $G/Z$ also evolve under the $\HCFp$. The quotient $G/Z$ is a $(k-1)$-step nilpotent Lie group, and so the $\HCFp$ seems well suited to being studied by induction on $k$.
	
	We employ Lauret's \emph{bracket flow} technique (see \cite{LauretBracketFlowRef2016}) in the proof of Theorem \ref{thm:nilpConvThm}. Instead of studying the $\HCFp$ directly, we investigate an equivalent ODE on the space of complex Lie brackets, which can be viewed as an algebraic variety in the space $\f g \otimes \Lambda^2 \f g^*$. Here, $\f g$ is the Lie algebra of $G$. The bracket flow isn't well-adapted to the aforementioned induction scheme, as the centre of the evolving brackets is not preserved. To overcome this, we follow the ideas in \cite[\S 3]{bohmLafImmRicFl} and \cite[\S 2]{arrLafHomPlu19}, studying the \emph{gauged bracket flow}. This is again equivalent to the $\HCFp$, but now preserves the centre of the evolving brackets. The key technical ingredient is then an appropraite monotone quantity, stationary on solitons, which can be applied inductively.
	
	Other similar results on the long-time behaviour of geometric flows (e.g. \cite{lauretHomNil01},\cite{HCFUni2019},\cite{arrLafHomPlu19},\cite{pujia2020positive},\cite{bohmLafImmRicFl}) also use the bracket flow. Here, the key steps in the analysis all involve viewing the bracket flow as the gradient flow of a certain functional coming from real geometric invariant theory. Except on $2$-step nilpotent Lie groups, the $\HCFp$ does not enjoy this property and so Theorem \ref{thm:nilpConvThm} is more elusive, requiring the induction program mentioned above.
	
	It is worth highlighting a key result on solitons needed to prove Theorem \ref{thm:nilpConvThm}. Namely, Theorem \ref{thm:nilpSolisAlg} where we show that the torsion-twisted Chern-Ricci tensor on any complex, nilpotent, left-invariant $\HCFp$ soliton $(G,g)$ satisfies $\Theta(g) = \lambda g + g(D\cdot,\cdot)$, where $\lambda \in \R$ and $D$ is a $g$-self-adjoint derivation of $\f g := \operatorname{Lie}(G)$ (cf. algebraic Ricci solitons \cite{jabRicSolAlg14}). This is a key step in the problem of existence and uniqueness of invariant nilpotent solitons.

	It is natural to ask what happens in the solvable case. To this end, we describe the long-time behaviour in a class of 2-step solvable Lie groups. Specifically, \emph{almost-abelian} complex Lie groups. We say a complex Lie group $(G,J)$ is almost-abelian if its Lie algebra admits a (complex) co-dimension one abelian ideal.
	\begin{thm}
		\label{thm:aaConvThm}
	Let $(G,J,g)$ be a simply-connected almost-abelian complex Lie group that is not nilpotent with left invariant Hermitian metric $g$. Let $g_t$ be the solution to the $\HCFp$ (\ref{eqn:HCF}) starting at $g_0 = g$. Then $g_t$ exists for all positive times and converges in the Cheeger-Gromov topology as $t\to \infty$ to a steady $\HCFp$ soliton $(G_\infty , J_\infty , g_\infty)$.
	\end{thm}
	By `converges', we mean that for any increasing sequence of times, there exists a subsequence on which the corresponding Hermitian manifolds converge. This is in contrast to Theorem \ref{thm:nilpConvThm}, where we show that there exists \emph{some} sequence of times for which convergence holds.
	
	The proof of Theorem \ref{thm:aaConvThm} again uses Lauret's bracket flow with a suitable monotone quantity. We note that similar results for the Ricci flow on almost-abelian Lie groups were obtained by Arroyo in \cite{ArrRicciAA}.

	Our final result addresses the existence and uniqueness of $\HCFp$-solitons on almost-abelian Lie groups. First note that a simply-connected, complex almost-abelian Lie group is completely determined up to isomorphism by a matrix $A \in \f{gl}_{n-1}(\C)$, where $n = \dim_{\C} (G,J)$. (see Section \ref{subsec:almostAbelian}). For $A \in \f{gl}_{n-1}(\C)$, denote by $(G_A,J_A)$ be the corresponding simply-connected complex almost-abelian Lie group.

	\begin{thm}
		\label{thm:solitonClass}
		Suppose $(G_A,J_A)$ is a simply-connected, almost-abelian complex Lie group. Then $(G_A,J_A)$ admits at most one left-invariant $\HCFp$ soliton up to homotheties. Moreover, $(G_A,J_A)$ admits a  left-invariant $\HCFp$ soliton if and only if $A$ is semi-simple or nilpotent.
	\end{thm}
	
	The proof of this theorem is constructive and thus yields examples of $\HCFp$-solitons in arbitrary dimensions (see Propositions \ref{prop:nilpSolCanForm} and \ref{prop:nonNilpAASolIsAlg}). Due to a theorem of Malcev, all nilpotent examples (which occur precisely when $A$ is nilpotent) admit co-compact lattices (see Remark \ref{rmk:lattices}).
	
	Let us mention other recent works studying HCFs on homogeneous spaces. In \cite{UstHCFHom2017}, Ustinovskiy studied the $\HCFp$ on complex homogeneous manifolds $G/H$. He showed that the space of (generically not $G$-invariant) metrics \emph{induced} by left-invariant metrics on $G$ is preserved, and that the flow in this case is governed by an ODE on the Lie algebra of $G$ which is independent of the isotropy $H$. In \cite{PanelliPodestaHCFCmpctHom2020}, Panelli and Podest\`a studied the $\HCFp$ on compact homogeneous spaces. And the $\HCFp$ on 2-step nilpotent Lie groups was studied by Pujia in \cite{pujia2020positive}.
	
	In their original paper, Streets and Tian chose $Q$ such that (\ref{eqn:HCF}) is a gradient flow of a certain functional when $M$ is compact. This flow was studied by Lafuente, Pujia and Vezzoni on complex unimodular Lie groups in \cite{HCFUni2019}. They showed that the flow exists for all positive times and after normalizing, converges in the Cheeger-Gromov topology to a soliton. They also studied the existence and uniqueness of invariant static metrics in this setting. The structure of solitons of this `gradient' HCF on complex Lie groups was studied by Pujia in \cite{PujiaSolitons2019}. Pediconi and Pujia also studied the same flow on locally homogeneous complex surfaces in \cite{Pediconi2020}.
	
	In \cite{StreetsTianPluriclosed2010}, Streets and Tian introduced the \emph{pluriclosed flow} by choosing $Q$ to preserve the pluriclosed condition on Hermitian metrics: $\partial \bar \partial \omega = 0$ where $\omega = g(J\cdot,\cdot)$. The pluriclosed flow was studied by Enrietti, Fino and Vezzoni on $2$-step nilmanifolds in \cite{pluriclosedTwoStep2015}, where it was shown that the flow exists for all positive times. Arroyo and Lafuente then showed in \cite{arrLafHomPlu19} that after a suitable normalisation, the flow converges to a soliton in this setting. They also studied the limiting behaviour of the flow on almost-abelian Lie groups. In \cite{bolingPCFLocHom}, Boling also studied the long-time existence and behaviour of the pluriclosed flow on locally homogeneous complex surfaces.
	
	The rest of the article is organised as follows. In section \ref{sec:prelims} we outline basic computations, discuss the bracket flow and $\HCFp$-solitons on complex Lie groups. In section \ref{sec:nilp} we analyse the $\HCFp$ on nilpotent Lie groups, studying solitons (Theorem {\ref{thm:nilpSolisAlg}}), growth behaviour (Theorem \ref{thm:BFNilpGaugedGrowth}), and proving Theorem \ref{thm:nilpConvThm}. In section \ref{sec:almostAbelian}, we discuss in detail almost-abelian complex Lie groups and prove theorem \ref{thm:aaConvThm}. Finally, in section \ref{sec:solitons} we study solitons on almost-abelian Lie groups and prove Theorem \ref{thm:solitonClass}.
	
	\subsubsection*{Acknowledgements}
	I am grateful to my advisor Ramiro Lafuente for his invaluable input and guidance. I also wish to thank Artem Pulemotov and Romina Arroyo for helpful discussions. This work was supported by an Australian Government Research Training Program (RTP) Scholarship.
	\section{Preliminaries}
	\label{sec:prelims}
	\subsection{$\HCFp$ on complex Lie groups}
	Let $G$ be a complex $n$-dimensional Lie group. Equivalently, $G$ is a $2n$ dimensional real Lie group admitting a bi-invariant complex structure $J$. Denote by $e$ the identity of $G$, $\f{g} \cong T_eG$ the Lie algebra, and $\mu \in \f g\otimes \Lambda^2 \f g^*$, the corresponding Lie bracket. Furthermore, denote $\f g^{1,0} := \{X\in\f g \otimes \C : JX = iX\}$. Let $g$ be a left-invariant, Hermitian metric on $(G,J)$. It was shown in \cite{HCFUni2019} that for a holomorphic, left-invariant frame $\{Z_1,\dots ,Z_{n}\} \subset \f g^{1,0}$, the torsion-twisted Chern-Ricci tensor $\Theta(g)$ is given by
	\begin{equation}
	\label{eqn:thetaDef}
	\Theta(g)(Z_i,Z_{\B j}) = \frac 12 g^{k\B l}g^{r \B s}g(Z_i,\mu(Z_{\B l},Z_{\B s}))\cdot g(\mu(Z_{k},Z_{r}) , Z_{\B j}),
	\end{equation}
	where we are employing the Einstein summation convention. We note here, as was done in \cite{HCFUni2019}, that the `$S$' term in $\Theta$ does not appear, since the mixed brackets $\mu(Z_i,Z_{\bar{j}})$ vanish on a complex Lie group.

	\subsection{The bracket flow}
	\label{sec:bracketFlow}
	The philosophy of varying brackets rather than metrics underlies a lot of successful research in homogeneous geometry. The \emph{bracket flow} approach introduced in \cite{lauret2015CurvatureFlows} for almost hermitian Lie groups stems from this idea and has been used to study a variety of geometric flows. We use this section to give a brief outline of it. 
	
	Let $(G,J)$ be a simply-connected complex Lie group with Lie algebra $(\f g,\mu)$ and fix a background, $J$-Hermitian, left-invariant inner product $\bgm{\cdot}{\cdot}$ on $\f g$. By considering $(\f g,J)$ as a fixed complex vector space, we regard $\mu$ as the Lie algebra itself, viewed as a member of the \emph{variety of complex Lie algebras}:
	\[
	\vcla := \{\mu \in  \f g \otimes \Lambda^2 \f g^* : \mu \text{ satisfies the Jacobi identity and }\mu(J\cdot,\cdot) = J\mu(\cdot,\cdot) \}.
	\]
	The Lie group
	\[
	\textrm{GL}(\f g,J) := \{f \in \textrm{GL}(\f g): f\circ J = J \circ f\} \cong \textrm{GL}_n(\C),
	\]
	acts on $\vcla$ via the `change of basis' action given by
	\begin{equation}
	\label{eqn:GLact}
	h \cdot \mu = h\mu(h^{-1}\cdot,h^{-1}\cdot), \qquad h \in \textrm{GL}(\f g,J).
	\end{equation}
	We can write any other Hermitian inner product on $\f g$ as $\bgm{\varphi\cdot}{\varphi\cdot}$ for some $\varphi \in \textrm{GL}(\f g, J)$. Considered as a map of Lie algebras,
	\[
	\varphi:(\f g, J, \mu, \bgm{\varphi\cdot}{\varphi\cdot}) \to (\f g, J, \varphi\cdot\mu,\bgm{\cdot}{\cdot})
	\]
	preserves the complex, metric, and Lie bracket structures. At the Lie group level, let $g$ be the left-invariant hermitian metric on $(G,J)$ defined by $g(e) = \bgm{\varphi\cdot}{\varphi\cdot}$. Let $(G_{\varphi\cdot \mu},J_{\varphi\cdot \mu},g_{\varphi\cdot \mu})$ be the unique simply-connected, complex Lie group with Hermitian metric defined by the data $(\f g, J, \varphi\cdot\mu,\bgm{\cdot}{\cdot})$. The Lie group homomorphism
	\[
	\varPhi: (G,J,g)\to (G_{\varphi\cdot \mu},J_{\varphi\cdot \mu},g_{\varphi\cdot \mu}),
	\]
	defined by $(\textrm{d}\varPhi)_e = \varphi$ is then an equivariant biholomorphic isometry. Thus, the space of left-invariant Hermitian metrics on $(G,J)$ can be parametrised by the orbit
	\[
	\textrm{GL}(\f g, J)\cdot \mu \subset \vcla.
	\]
	Define $\textrm{U}(\f g,J) := \{k \in \textrm{GL}(\f g,J) : \bgm{k\cdot}{k\cdot} = \bgm{\cdot}{\cdot} \}$. Then, two brackets $\mu_1,\mu_2 \in \textrm{U}(\f g, J)\cdot \mu$ correspond to biholomorphically isometric Lie groups. 
	From this `varying brackets' view point, we can naturally ask what the $\HCFp$ looks like in the orbit $\textrm{GL}(\f g, J)\cdot \mu$. Define
	\[
	\f {gl}(\f g,J) := \{A \in \operatorname{End}(\f g) : [A,J] = 0\} \cong \f{gl}_n(\C),
	\]
	to be the Lie algebra of $\operatorname{GL}(\f g,J)$. The \emph{bracket flow} is the following ODE of brackets in the space $\vcla \subset   \f g \otimes \Lambda^2 \f g^*$.
	\begin{equation}
	\label{eqn:bracketFlow}
	\dot \mu_t = -\pi(P_{\mu_t})\mu_t,\qquad \mu_0 = \mu.
	\end{equation}
	Here $\pi$ is the derivative of the action (\ref{eqn:GLact}) given by 
	\[
	\pi(A)\mu = A\mu - \mu(A\cdot , \cdot) - \mu(\cdot , A \cdot), \qquad A \in \f {gl}(\f g , J),
	\]
	and the map $\vcla \ni \mu \mapsto P_\mu \in \f{gl}(\f g, J)$ is defined implicitly by
	\begin{equation}
	\label{eqn:pmuDef}
	g(P_{\mu}\cdot , \cdot) = \Theta_\mu(g),
	\end{equation}
	Where $\Theta_\mu(g)$ is the torsion-twisted Chern Ricci tensor associated to the Hermitian Lie group defined by $(\f g,J,\mu,\bgm{\cdot}{\cdot})$. The bracket flow is equivalent to the $\HCFp$ in the following sense.
	\begin{thm}[{\cite[Theorem 1.1]{lauret2015CurvatureFlows}}]
	\label{thm:BracketFlowEquiv}
	Let $(G,J,g)$ be a complex Lie group with left-invariant metric $g$. Let $g_t$ be the solution to the $\HCFp$ (\ref{eqn:HCF}) and $\mu_t$ the solution to the bracket flow (\ref{eqn:bracketFlow}). Then $g_t$ and $\mu_t$ are defined on the same time-interval $0\in I \subset \R$. Moreover, for all $t \in I$, the Lie group with left-invariant Hermitian metric $(G_{\mu_t},J_{\mu_t},g_{\mu_t})$ defined by the data $(\f g,J,\mu_t,\bgm{\cdot}{\cdot})$ is equivariantly, biholomorphically isometric to $(G,J,g_t)$.
	\end{thm}
	
	More generally, the $\HCFp$ is also equivalent to the so called \textit{gauged bracket flow}, which comes from the fact that $P_{\mu}$ is $\operatorname{U}(\f g,J)$-equivariant (i.e. $P_{k\cdot \mu} = kP_\mu k^*$ for $k\in \operatorname{U}(\f g,J)$). First recall that the Lie algebra of $\operatorname{U}(\f g, J)$ is given by
	\[
	\f u(\f g, J) = \{A \in \f{gl}(\f g, J): \bgm{A\cdot}{\cdot}= - \bgm{\cdot}{A\cdot}\}.
	\]
	Then we have the following
	\begin{thm}
		\label{thm:gaugedBracketFlowEquiv}
		Suppose $\nu_t$ solves the gauged bracket flow equation
		\begin{equation}
		\label{eqn:gaugedBF}
		\dot \nu_t = -\pi(P_{\nu_t} - S_{\nu_t})\nu_t,\qquad\nu_0 = \mu,
		\end{equation}
		where $\vcla \ni \nu \mapsto S_\nu \in \f{u}(\f g,J)$ is a smooth map. Then there exists a smooth family $k_t \in \operatorname{U}(\f g, J)$ such that $\nu_t = k_t \cdot \mu_t$, where $\mu_t$ solves the bracket flow (\ref{eqn:bracketFlow}). In particular, $(G_{\mu_t},J_{\mu_t},g_{\mu_t})$ and $(G_{\nu_t},J_{\nu_t},g_{\nu_t})$ are equivariantly, biholomorphically isometric.
	\end{thm}
	\begin{proof}
		The proof precisely follows that of \cite[Theorem 2.2]{arrLafHomPlu19}, where $k_t$ is chosen to solve $\dot k_t = S_{k_t\cdot \mu_t}k_t$ with initial condition $k_0 = \id$.
	\end{proof}
	
	By choosing a smooth map $\vcla \ni \nu \mapsto S_\nu \in \f{u}(\f g , J)$ appropriately, analysis of the gauged bracket flow (\ref{eqn:gaugedBF}) can be greatly simplified.
	
	As a final modification to the bracket flow, we may wish to observe the behaviour of the dynamical system after an appropriate normalisation. That is, the behaviour of $\tilde \nu_t := c(t)\nu_t$ for some positive function $c$. After a time re-parametrization, $\tilde \nu_t$ satisfies
	\begin{equation}
	\tilde \nu_t = -\pi(P_{\tilde \nu_t} - S_{\tilde \nu_t} + r(t) \operatorname{Id}){\tilde \nu_t} \qquad \tilde \nu_0 = c(0)\mu, 
	\end{equation}
	for some appropriately defined function $r\colon \R\to \R$ (see \cite[\S 3.3]{lauretRicFlowHomMfd}).
	
	Let us compute $P_\mu$ explicitly as we will refer to it later.
	\begin{lem}
		\label{lem:PmuForm}
		Let $\{Z_i\}_{i=1}^n$ be a left-invariant $g$-unitary frame on $(G,J,g)$, then $P_\mu$ is given by
		\[
		P_\mu X = \sum_{i < j} g(X, \mu(Z_{\B i},Z_{\B{j}}))\mu(Z_i,Z_j)
		\]
		for each $X \in \f g^{1,0}$.
	\end{lem}
	\begin{proof}
		This follows directly from formula (\ref{eqn:thetaDef}), the definition of $P_\mu$ (\ref{eqn:pmuDef}), and the fact that $\{Z_i\}_{i=1}^n$ is $g$-unitary.
	\end{proof}
	
	The background inner product $\bgm{\cdot}{\cdot}$ on $\f g$ naturally induces an inner product on $\Lambda^2 \f g$ defined by $\bgm{v_1\wedge v_2}{w_1\wedge w_2} := \det(\bgm{v_i}{w_j})$ for $v_i,w_i \in \f g$, $1 \leq i \leq 2$. It immediately follows that $\{Z_i\wedge Z_i\}_{i<j}$ is a unitary basis for this inner product. Denote by $(\cdot)^*$ the adjoint of a linear map between two inner product spaces. We can also induce an inner product on all tensor combinations of $\f g$, $\f g^*$ and $\Lambda^2 \f g$ in the natural way. In particular, the induced inner product on $\f g\otimes \f g^* \cong \operatorname{End}(\f g)$ is given by $\bgm{A}{B} = \tr(AB^*)$ for $A,B \in \operatorname{End}(\f g)$. Similarly, the inner product on the space of brackets $\f g\otimes \Lambda^2 \f g^*$ is given by $\bgm{\mu}{\nu} = \tr(\mu\nu^*)$ for $\mu,\nu \colon \Lambda^2 \f g \to \f g$.
	
	In this notation, Lemma \ref{lem:PmuForm} gives a simple formula for $P_\mu$.
	\begin{cor}
	\label{cor:PmuAdjForm}
	$P_\nu = \nu\nu^*$, for all brackets $\nu \in \f g \otimes \Lambda^2 \f g^*$.
	\end{cor}
	\begin{proof}
	Let $\{Z_i\}_{i=1}^n\subset \f g^{1,0}$ be a unitary basis for $\f g^{1,0}$, then $\{Z_i\wedge Z_j\}_{i < j}$ is a unitary basis for $\Lambda^2 \f g^{1,0}$. 	With this, we note that by Lemma \ref{lem:PmuForm}, for $X \in \f g^{1,0}, \nu \in   \f g \otimes \Lambda^2 \f g^*$,
	\[
	\begin{split}
	P_\nu X &= \sum_{i < j} \bgm{X}{\B{\nu(Z_i \wedge Z_j)}}\nu(Z_i \wedge Z_j)\\
	&= \sum_{i < j} \bgm{\nu^*X}{\B {Z_i \wedge Z_j}}\nu(Z_i \wedge Z_j)\\
	&= \nu \nu^*X.
	\end{split}
	\]
	Thus, $P_\nu = \nu \nu^*$ for all $\nu \in  \f g \otimes \Lambda^2 \f g^*$.
	\end{proof}
	
	\subsection{Static metrics and solitons}
	Important solutions to (\ref{eqn:HCF}) are the so called static metrics and solitons. A Hermitian metric $g$ on $(M,J)$ is called an ($\HCFp$)-static metric if
	\[
	\Theta(g) = \lambda g,
	\]
	for some $\lambda \in \R$. Such metrics are special in the sense that the corresponding $\HCFp$ solution evolves only by scaling of the initial metric. These are analogous to Einstein metrics in the Riemannian case. We can observe the following necessary condition for left-invariant static metrics to exist on a complex Lie group $G$.
	
	\begin{prop}
		If $(G,J)$ is a non-abelian complex Lie group admitting a left-invariant static metric, then $\mu(\f g,\f g) = \f g$. In particular, if $G$ is solvable, it can not admit a static left-invariant metric. 
	\end{prop}
	\begin{proof}
	The proof here follows that of \cite[Corollary 3.2]{pujia2020positive}. Suppose $(G,J)$ admits a left-invariant static metric $g$. If $\mu(\f g,\f g) \neq \f g$ then we may choose a $g$-unitary basis $\{Z_1 \dots Z_n\}$ such that $Z_1 \perp \mu(\f g,\f g)$. Then we see that,
	\[
	\lambda g(Z_1,Z_{\B 1}) = \Theta(g)(Z_1,Z_{\B 1}) = \sum_{i<j}|g(Z_1,\mu(Z_{\B i},Z_{\B j}))|^2 = 0.
	\]
	Thus, $\lambda = 0$ and so $\Theta(g) = 0$. But then by Corollary \ref{cor:PmuAdjForm},
	\[
	0 = \operatorname{tr}P_\mu = \|\mu\|^2.
	\]
	So $\mu = 0$, but $(G,J)$ is non-abelian which is a contradiction.
	\end{proof}
	
	More generally, we define an $\HCFp$-soliton as a metric $g$ satisfying
	\[
	\Theta(g) = \lambda g + \mathcal{L}_Zg,
	\]
	where $\lambda \in \mathbb{R}$ and $Z \in \Gamma(M,T^{1,0}M)$ is a holomorphic vector field. Here $\mathcal{L}$ is the usual Lie derivative. Solitons are called \textit{expanding}, \textit{shrinking}, or \textit{steady} if $\lambda$ is negative, positive or $0$ respectively.
	
	If $(g_t)_{t\in I}$ solves (\ref{eqn:HCF}) for some interval $0 \in I \subset \mathbb{R}$ with $g_0 = g$. Then the evolution is given by a combination of scaling by some smooth positive function $c\colon I \to \mathbb{R}_{> 0}$ and pull-back by biholomorphisms $\phi_t\colon M \to M$. That is,
	\[
	g_t = c(t)\phi_t^*g.
	\]
	Thus, solitons are ``geometric fixed points" of the $\HCFp$ (\ref{eqn:HCF}). These are candidate attractors for the normalised $\HCFp$. 
	
	\begin{rmk}
	By scale, and biholomorphism invariance of $\Theta(g)$, one can see that $c(t) = 1-\lambda t$. Thus, any soliton solution of (\ref{eqn:HCF}) that exists for all positive times must be expanding or steady ($\lambda \leq 0$).
	\end{rmk}
	
	A complex Lie group $G$ with left-invariant Hermitian metric $g$ is called a \textit{semi-algebraic} $\HCFp$-soliton if
	\[
	\Theta(g) = \lambda g + \frac 12(g(D\cdot,\cdot) + g(\cdot,D\cdot)),
	\]
	for some $\lambda \in \R$ and $D \in \operatorname{Der}(\f g)$ such that $[D,J] = 0$.
	
	Semi-algebraic solitons are in particular solitons in the usual sense, but the corresponding biholomorphisms driving the evolution are also Lie automorphisms of $G$. More precisely, let $\operatorname{Aut}(G,J)$ denote the automorphisms of the complex Lie group $(G,J)$ (i.e. Automorphisms of the real group $G$ that are also biholomorphisms). One can show that $g_t = (1-\lambda t)\varphi_t^*g$, where for each $t$, $\varphi_t \in \operatorname{Aut}(G,J)$ is the unique automorphism satisfying $(d\varphi_t)_e =  \exp(tD)\in \operatorname{Aut}(\f g,J)$. Recall that $G$ being simply-connected means that $\varphi_t \in \operatorname{Aut}(G,J)$ is uniquely determined by $(d{\varphi_t})_e$.
	In general it is not known whether any left-invariant soliton is semi-algebraic, however this can be shown to be true in the nilpotent case (cf. \cite{lauretHomNil01}).
	
	\begin{prop}
		\label{prop:nilpSolSemiAlg}
		If a complex simply-connected nilpotent Lie group with left-invariant metric, $(G,J,g)$ is a soliton, then it is a semi-algebraic soliton.
	\end{prop}
	\begin{proof}
		The proof combines those in \cite[Prop 1.1]{lauretHomNil01} and \cite[Thm 3.1]{jabHomRicciSol15} but we state it here for convenience. Let $g_t$ solve (\ref{eqn:HCF}) starting at $g_0 = g$. Then $g_t = c(t)\phi_t^*g$. We may assume that $\phi_t$ fixes the identity $e$. If not, then by left-invariance of $g$, we may replace it with $L_{\phi_t(e)^{-1}}\circ \phi_t$. We claim that $\phi_t$ is an automorphism of $G$ for all $t$. Consider $G$ as a subset of $\operatorname{Isom}(G)$ by identifying $x\in G$ with $L_x \in \operatorname{Isom}(G)$. Then $G$ is a simply-connected nilpotent subgroup of $\operatorname{Isom}(G)$ acting transitively on $G$. Thus, $\phi_t G\phi_t^{-1} \subset \operatorname{Isom}(G)$ is also a simply-connected nilpotent subgroup and also acts transitively on $G$. It follows from \cite[Thm 2]{wilsonIsomGrpNil82} that $\phi_t G \phi_t^{-1} = G$. By evaluating at the identity we see that
		\[
		\phi_t L_x \phi_t^{-1} = L_{\phi_t(x)},
		\]
		and it follows that $\phi_t \in \operatorname{Aut}(G)$. Thus, since $\phi_0 = \id$, $D := (d\phi_0)_e $ must be a derivation. So evaluating at the identity,
		\[
		\Theta(g_0) = -\dot g_0 = -\dot c(0) g_0 - g_0(D\cdot,\cdot) - g_0(\cdot,D \cdot),
		\]
		and $(G,J,g = g_0)$ is a semi-algebraic soliton.
 	\end{proof}
	
	A semi-algebraic soliton is called an \textit{algebraic} $\HCFp$-soliton if $D\in \operatorname{Der}(\f g)$ can be chosen to be $g$-self-adjoint. That is,
	\[
	\Theta(g) = \lambda g + g(D\cdot,\cdot).
	\]
	In the varying brackets setting we can make the following definition.
	
	\begin{defn}
	A bracket $\mu \in \mathcal{C}$ is called a semi-algebraic soliton bracket if
	\[
	P_\mu = \lambda \operatorname{Id} + \frac 12(D + D^*),
	\] 
	for some $\lambda\in \mathbb{R}$ and $D \in \operatorname{Der}(\f g) \cap \f{gl}(\f g,J)$. $\mu$ is called \emph{algebraic} if $D = D^*$ or equivalently, if $D^* \in \operatorname{Der}(\f g)$.
	\end{defn}
	\begin{rmk}
		It is easy to see that $\mu \in \vcla$ is a semi-algebraic (resp. algebraic) soliton bracket if and only if the corresponding complex Lie group with left-invariant Hermitian metric $(G_\mu,J_\mu,g_\mu)$ is a semi-algebraic (resp. algebraic) soliton.
	\end{rmk}
	 
	Such a definition is useful because of the following
	\begin{prop}[\cite{arrLafHomPlu19}]
		\label{prop:algSol}
		A bracket $\mu \in \vcla$ is a semi-algebraic soliton if and only if it is a fixed point of some normalised, gauged bracket flow.
	\end{prop}
	
	\begin{rmk}
		
		Observe that $\ker({B \mapsto \pi(B)\mu}) = \operatorname{Der}(\mu)$, and so algebraic solitons are precisely the scale-static solutions of the bracket flow (\ref{eqn:bracketFlow}) (that is, with no gauging). Indeed $\mu$ is an algebraic soliton if and only if $\pi(P_\mu)\mu = -\lambda \mu$.
	\end{rmk}
	
	In the case of Ricci flow, it was shown in \cite{jabRicSolAlg14} that all semi-algebraic solitons are in fact algebraic. The same is also true for the $\HCFp$ on $2$-step nilpotent complex Lie groups as shown in \cite{pujia2020positive}.

	\section{$\HCFp$ on complex nilpotent Lie groups}
	\label{sec:nilp} 
	In this section we collect results for the $\HCFp$ on complex simply-connected nilpotent Lie groups to prove our main result. First, we show that all left-invariant solitons in this setting are algebraic (see Theorem \ref{thm:nilpSolisAlg}). Next, we investigate the growth behaviour of the bracket flow in Theorem \ref{thm:BFNilpGaugedGrowth}. Finally, we analyse the bracket flow with suitable gauging and introduce an appropriate monotone quantity to prove Theorem \ref{thm:nilpConvThm} inductively.
	
	Let $(G^{2n},J,g)$ be a simply-connected, complex, nilpotent Lie group with left-invariant Hermitian metric $g$. Denote by $(\f g,\mu)$ the Lie algebra of $G$ with bracket $\mu 
	\in \f g \otimes \Lambda^2 \f g^*  $. As in Section \ref{sec:bracketFlow}, we define $\bgm{\cdot}{\cdot} := g(e)$, an inner product on $\f g$.
	
	To simplify notation later on, recall that we can naturally define the wedge product of endomorphisms, $\wedge \colon \f{gl}(\f g) \times \f{gl}(\f g) \to  \f {gl}(\Lambda^2 \f g)$, by
	\[
	A\wedge B(v\wedge w) := Av \wedge Bw + Bv \wedge Aw,
	\]
	for $v,w \in \f g$. Notice that this construction is linear in both arguments and symmetric. It is easily seen that
	\[
	(A\wedge B)\circ (C \wedge D) = AC \wedge BD + AD \wedge BC,
	\]
	for $A,B,C,D \in \f {gl}(\f g)$.

	One can also check that for $A,B \in \f {gl}(\f g)$, the adjoint of $A\wedge B \colon \Lambda^2 \f g \to \Lambda^2 \f g$ with respect to the natural inner product on $\Lambda^2 \f g$ is given by
	\[
	(A\wedge B)^* = A^* \wedge B^*.
	\]
	With this notation, the representation $\pi \in \f{gl}( \f g  \otimes \Lambda^2 \f g^* ) \otimes \f {gl}(\f g)^*$ defined in Section \ref{sec:bracketFlow} is given by
	\[
	\pi(A)\nu = A\nu - \nu \id \wedge A,
	\]
	for $A \in \f {gl}(\f g), \nu \in   \f g \otimes \Lambda^2 \f g^*$. 
	
	 For the rest of this section, we will consider the unitary decomposition
	\[
	\f g = \f z \oplus \f z^\perp,
	\]
	where $\f z = \f z(\mu)$ is the centre of $(\f g,\mu)$. Note that $\f z \neq \{0\}$ as $\mu$ is nilpotent. For a subspace $\f v \subset \f g$, let $\Pr_{\f v}\colon \f g \to \f v \subset \f g$ denote the orthogonal projection onto $\f v$ via the background metric $\bgm{\cdot}{\cdot}$. Using this we may write
	\[
	\mu = \mu_0 + \mu_1,
	\]
	where the brackets $\mu_0,\mu_1 \in   \f g \otimes \Lambda^2 \f g^*$ are defined by $\mu_0 := \Pr_{\f z}\mu$ and $\mu_1 := \Pr_{\f z^\perp} \mu$. Note that $\mu_0$ is a two step nilpotent Lie bracket, while the Lie algebra $(\f z^\perp,\mu_1|_{\f z^\perp \times \f z^\perp})$ is isomorphic to the quotient $(\f g,\mu)/\f z(\mu)$. Thus, $\mu_1$ has degree of nilpotency one less than that of $\mu$. This fact will be frequently used in the proofs to follow.
	
	\subsection{Nilpotent $\HCFp$ solitons}
	To begin proving Theorem \ref{thm:nilpConvThm}, we first need to characterise nilpotent $\HCFp$ solitons. Specifically, we will prove Theorem \ref{thm:nilpSolisAlg}. That is, all left-invariant solitons are algebraic. By doing this, we will see that fixed points of any normalised gauged bracket flow are also fixed points of the normalised bracket flow with no gauging.
	
	For this section, let $\{e_i\}_{i=1}^{n}\cup \{e_i:=Je_{i-n}\}_{i=n+1}^{2n}\subset \f g$ be an orthonormal basis of $\f g$. Moreover, for a bracket $\nu \in \f g \otimes \Lambda^2 \f g^*$, denote by 
	\[
		\operatorname{Der}(\nu) := \{D \in \f{gl}(\f g,J): D\nu(\cdot,\cdot) = \nu(D \cdot,\cdot) + \nu(\cdot,D\cdot)\}
	\]
	its set of derivations.
	
	Recall that a semi-algebraic soliton bracket $\mu \in \f g\otimes \Lambda^2 \f g^*$ is one which satisfies
	
	\[
	P_{\mu} = \lambda \id + \frac{1}{2}(D + D^*),
	\]
	where $D \in \operatorname{Der}(\mu)$. Our goal is to show that $D^* \in \operatorname{Der}(\mu)$ as well.
	
	We start with the most important lemma for this proof inspired from real geometric invariant theory (see the proof of \cite[Proposition 3.4]{HCFUni2019}).
	
	\begin{lem}
		\label{lem:GITcondn}
		Suppose $\mu$ is a semi-algebraic soliton bracket. That is,
		\[
		P_\mu = \lambda \id + \frac 12 (D+D^*),
		\]
		for $\lambda \in \R$ and $D \in \operatorname{Der}(\mu)$. Then,
		\[
		0 = \|\pi(D^*)\mu\|^2 + \operatorname{tr}(Q_\mu[D,D^*]),
		\]
		where $Q_\mu \in \f{gl}(\f g,J)$ is defined as
		\[
		\bgm{Q_\mu v}{w} := \sum_{i} \bgm{\mu(v\wedge e_i)}{\mu(w\wedge e_i)},
		\]
		for $v,w \in \f g$.
	\end{lem}
	\begin{proof}
		First we claim that for all $E \in \f{gl}(\f g)$,
		\[
		\tr((P_\mu - Q_\mu)E) = \langle \pi(E)\mu,\mu\rangle. 
		\]
		This is a well known fact related to the moment map of the $\operatorname{GL}(\f g)$ action on $\f g \otimes \Lambda^2 \f g^*$ from real geometric invariant theory (see e.g. \cite{bohm2017real}). But it is simple to see since
		\[
		\begin{split}
		\tr((P_\mu - Q_\mu)E) =& \sum_{k,i<j}\bgm{Ee_k}{\mu(e_i\wedge e_j)}\bgm{\mu(e_i\wedge e_j)}{e_k}\\
		&- \sum_{ik} \bgm{\mu(Ee_k \wedge e_i)}{\mu(e_k \wedge e_i)}\\
		=& \sum_{i<j}\bgm{E\mu(e_i\wedge e_j)}{\mu(e_i\wedge e_j)}\\
		&-\sum_{i < j} \bgm{\mu \id \wedge E (e_i\wedge e_j)}{\mu(e_i \wedge e_j)}\\
		=& \bgm{\pi(E)\mu}{\mu}.
		\end{split}
		\]
		Thus, on one hand since $P_\mu = \lambda \id + \frac{1}{2}(D+D^*)$,
		\[
		\tr (P_\mu - Q_\mu)[D,D^*] = -\tr Q_\mu[D,D^*].
		\]
		On the other hand
		\[
		\begin{split}
		\tr(P_\mu - Q_\mu)[D,D^*] =& \bgm{\pi([D,D^*])\mu}{\mu}\\
		=& \bgm{\pi(D)\pi(D^*)\mu}{\mu}\\
		=& \|\pi(D)^*\|^2,
		\end{split}
		\]
		where in the second line we used that $\pi$ is a Lie algebra representation and in the third we used the easily verifiable fact that $\pi(D^*) = \pi(D)^*$. The result is now clear.
	\end{proof}
	Notice that the quantity $\|\pi(D^*)\mu\|^2$ measures the failure of $D^*$ being a derivation. The problematic term is therefore one involving $Q_\mu$. We will handle this by induction. Let us now characterise derivations under the splitting $\f g = \f z \oplus \f z^\perp$.
	\begin{prop}
		\label{prop:derBlock}
		$D \in \operatorname{Der}(\mu = \mu_0 + \mu_1)$ if and only if
		\[
		D = \begin{pmatrix}
		D_{00}&D_{01}\\
		0&D_{11}
		\end{pmatrix},
		\]
		for maps $D_{00} \colon \f z \to \f z$, $D_{01}\colon \f z^{\perp} \to \f z$ and $D_{11}\colon \f z^\perp \to \f z^\perp$ satisfying
		\[
		D_{00}\mu_0 + D_{01}\mu_1 - \mu_0 \id \wedge  D_{11} = 0 \qquad \text{ and } \qquad D_{11}\in \operatorname{Der}(\mu_1).
		\]
	\end{prop}
	\begin{proof}
		If $D \in \operatorname{Der}(\mu)$ then $D \f z \subset \f z$ since $\mu(D \f z,\f g) = D\mu(\f z,\f g) - \mu(\f z,D \f g) = 0$. So indeed $D$ has the block form above. To see the relations between the blocks we simply observe that
		\[
		\begin{split}
		0 &= \pi(D)\mu \\
		&= \pi(D)\mu_0 + \pi(D)\mu_1\\
		&= D_{00}\mu_0 - \mu_0\id \wedge D_{11}+ D_{01}\mu_1 + \pi(D_{11})\mu_1,
		\end{split}
		\]
		equating the $\f z$ and $\f z^\perp$ parts immediately implies the result.
	\end{proof}
	\begin{cor}
		\label{cor:solitonBlock}
		$\mu = \mu_0 + \mu_1$ is a semi-algebraic soliton bracket if and only if there is a real number $\lambda \in \mathbb{R}$ and maps $D_{00} \colon \f z \to \f z$, $D_{01}\colon \f z^\perp \to \f z$ and $D_{11}\colon \f z^\perp \to \f z^\perp$ satisfying
		
		\begin{enumerate}
			\item[(i)] $D_{11}\in \operatorname{Der}(\mu_1)$,
			\item[(ii)] $D_{00}\mu_0 + D_{01}\mu_1 - \mu_0 \id \wedge  D_{11} = 0$,
			\item[(iii)] $\mu_i\mu_i^* = \lambda \id + \frac{1}{2}(D_{ii}+D_{ii}^*)$, $i \in \{0,1\}$ and,
			\item[(iv)] $\mu_0\mu_1^* = \frac{1}{2}D_{01}$
		\end{enumerate}
	\end{cor}
	\begin{proof}
		This follows from the fact that $P_\mu = \mu\mu^*$ (Corollary \ref{cor:PmuAdjForm}).
	\end{proof}
	Now we can rewrite Lemma \ref{lem:GITcondn} in terms of derivations written in block form with respect to the splitting $\f g = \f z \oplus \f z^\perp$.
	\begin{lem}
		\label{lem:GITBlock}
		If $\mu = \mu_0 + \mu_1$ is a semi-algebraic soliton for $\lambda \in \mathbb{R}$ and $D = \begin{pmatrix}
		D_{00}&D_{01}\\
		0&D_{11}
		\end{pmatrix}\in \operatorname{Der}(\mu)$, then
		\[
		0 = \|D_{00}^*\mu_0 - \mu_0 \id \wedge D_{11}^*\|^2 + \|D_{01}^*\mu_0 + \pi(D_{11}^*)\mu_1\|^2 + \bgm{\mu \id \wedge [D_{11},D_{11}^*]}{\mu}
		\]
	\end{lem}
	\begin{proof}
		We will use Lemma \ref{lem:GITcondn}. First,
		\[
		\begin{split}
		\pi(D^*)\mu =& \pi \begin{pmatrix}
		D_{00}^*&0\\
		D_{01}^*&D_{11}^*
		\end{pmatrix}(\mu_0 + \mu_1) \\
		=& D_{00}^*\mu_0 - \mu_0 \id \wedge D_{11}^*\\
		&+ D_{01}^*\mu_0 + D_{11}^*\mu_1 - \mu_1 \id \wedge D_{11}^*\\
		&- \mu_0 \id \wedge D_{01}^* - \mu_1\id \wedge D_{01}^*. 
		\end{split}
		\]
		Thus,
		\[
		\begin{split}
		\|\pi(D^*)\mu\|^2 =& \|D_{00}^*\mu_0 - \mu_0 \id \wedge D_{11}^*\|^2 + \|D_{01}^*\mu_0 + D_{11}^*\mu_1 - \mu_1 \id \wedge D_{11}^*\|^2\\
		&+\|\mu_0\id \wedge D_{01}^*\|^2 + \|\mu_1\id\wedge D_{01}^*\|^2.
		\end{split}
		\]
		Now, notice that $Q_\mu \f z = 0$, $Q_\mu \f z^\perp \subset \f z^\perp$, and
		\[
		[D,D^*]=\begin{pmatrix}
		*&*\\
		*&[D_{11},D_{11}^*]-D_{01}^*D_{01}.
		\end{pmatrix}
		\]
		So,
		\[
		\begin{split}
		\tr(Q_\mu [D,D^*]) =& \tr(Q_\mu [D_{11},D_{11}^*])-\tr(Q_\mu D_{01}^*D_{01})\\
		=& \tr(Q_\mu[D_{11},D_{11}^*]) - \sum_{ij} \bgm{\mu(D_{01}^*D_{01}e_i\wedge e_j)}{\mu(e_i \wedge e_j)}\\
		=& \tr(Q_\mu[D_{11},D_{11}^*]) - \sum_{i<j} \bgm{\mu^*\mu(D_{01}^*D_{01}e_i\wedge e_j)}{e_i \wedge e_j}\\
		&- \sum_{i<j} \bgm{\mu^*\mu(D_{01}^*D_{01}e_j\wedge e_i)}{e_j \wedge e_i}\\
		=&  \tr(Q_\mu[D_{11},D_{11}^*]) - \tr(\mu^*\mu \id \wedge D_{01}^*D_{01})\\
		=& \tr(Q_\mu[D_{11},D_{11}^*]) - \tr(\mu^*\mu \id \wedge D_{01}^*\id\wedge D_{01})\\
		&+ \tr(\mu^*\mu D_{01}\wedge D_{01}^*)\\
		=& \tr(Q_\mu[D_{11},D_{11}^*]) - \|\mu \id \wedge D_{01}^*\|^2\\
		=& \tr(Q_\mu[D_{11},D_{11}^*]) - \|\mu_0 \id \wedge D_{01}^*\|^2 - \|\mu_1 \id \wedge D_{01}^*\|^2,
		\end{split}
		\]
		where we have used the fact that $A\wedge B \circ C \wedge D = AC\wedge BD + AD\wedge BC$ and $(A\wedge B)^* = A^* \wedge B^*$. Adding the two terms, the result now easily follows from Lemma \ref{lem:GITcondn}.
	\end{proof}
	
	\begin{thm}
		\label{thm:nilpSolisAlg}
		Any complex, nilpotent, left-invariant soliton $(G,J,g)$ is algebraic.
	\end{thm}
	\begin{proof} By Proposition \ref{prop:nilpSolSemiAlg}, all left-invariant solitons are semi-algebraic. Thus, it suffices to show that all semi-algebraic solitons are algebraic. Suppose $\mu \in \vcla$ is nilpotent and satisfies
		\[
		P_\mu = \lambda \id + \frac 12 (D + D^*),
		\]
		for $\lambda \in \R$ and $D \in \operatorname{Der}(\mu)$. We will show that $D^*\in \operatorname{Der}(\mu)$. The proof will be by induction on the degree of nilpotency of $\mu$. The base case when $\mu = 0$ is trivial, since $\operatorname{Der}(0) = \f{gl}(\f g)$.	
		
		Now suppose that the conclusion holds for all $(k-1)$-step nilpotent brackets for some $k \geq 2$, and suppose $\mu = \mu_0 + \mu_1$ is $k$-step nilpotent. Then, $\mu_1$ is $(k-1)$-step nilpotent and by Corollary \ref{cor:solitonBlock}, it satisfies
		\[
		P_{\mu_1} = \lambda \id + \frac 12 (D_{11} + D_{11}^*),
		\]
		for $D_{11}\in \operatorname{Der}(\mu_1)$. Thus, by the inductive hypothesis, $D_{11}^* \in \operatorname{Der}(\mu_1)$. We see that $D_{11}$ is also normal. Indeed, since
		\[
		0 = \pi(D_{11})\mu_1 = \pi(D_{11}^*)\mu_1,
		\]
		we have
		\[
		\begin{split}
		[D_{11},D_{11}^*] &= 2[D_{11},P_{\mu_1}]\\
		&= 2D_{11}\mu_1\mu_1^* - 2\mu_1\mu_1^*D_{11}\\
		&= 2\mu_1 \id \wedge D_{11} \mu_1^* - 2\mu_1 (D_{11}^*\mu_1)^*\\
		&= 2\mu_1 \id \wedge D_{11} \mu_1^* - 2 \mu_1(\mu_1\id \wedge D_{11}^*)^*\\
		&= 0,
		\end{split}
		\]
		where we used the fact that $\pi(A)\nu = A\nu - \id \wedge A$. Thus, by Lemma \ref{lem:GITBlock},
		\[
		\|D_{00}^*\mu_0 - \mu_0 \id \wedge D_{11}^*\|^2 + \|D_{01}^*\mu_0 + \pi(D_{11}^*)\mu_1\|^2 = 0.
		\]
		Since $D_{11}\in \operatorname{Der}(\mu_1)$, we get that
		\[
		D_{00}^*\mu_0 - \mu_0 \id \wedge D_{11}^* = 0,
		\]
		and,
		\[
		D_{01}^*\mu_0 = 0.
		\]
		So $0 = \mu_0^*D_{01} = 2 \mu_0 \mu_0^* \mu_1,$
		which implies $D_{01} = 2 \mu_0^*\mu_1 = 0$ since $\ker \mu_0^*\mu_0 = \ker \mu_0$. Thus, by Proposition \ref{prop:derBlock}, $D^*$ is a derivation of $\mu$ and the result follows by induction.
	\end{proof}
	
	\subsection{Long time behaviour of the $\HCFp$ on nilpotent Lie groups}
	
	We will now investigate the asymptotic behaviour of the bracket flow on nilpotent Lie groups. We first describe the growth behaviour of brackets evolving under the bracket flow. In particular, we will see that they converge to $0$. Then, we prove Theorem \ref{thm:nilpConvThm}. Both proofs will use induction on the degree of nilpotency of the brackets involved.
	
	Recall that under the splitting $\f g = \f z \oplus \f z^\perp$, we write
	\[
	\mu = \mu_0 + \mu_1,
	\]
	where $\mu_0 = \Pr_{\f z}\mu$ and $\mu_1 = \Pr_{\f z^\perp}\mu$. By Corollary \ref{cor:PmuAdjForm},
	\[
	P_\mu = \mu_0\mu_0^* + \mu_0\mu_1^* + \mu_1\mu_0^* + \mu_1\mu_1^*.
	\]
	Notice that $P_\mu$ does not preserve the centre, and so in general, the bracket flow will not preserve the splitting $\f z \oplus \f z^\perp$. To overcome this, we introduce the following gauge map
	\[
	S_\nu := \nu_1\nu_0^* - \nu_0\nu_1^*
	\]
	for $\nu \in \vcla$.
	\begin{lem}
		\label{lem:gauging}
		$S_{\nu} \in \f u(\f g,J)$ for all $\nu \in \vcla$.
	\end{lem}
	\begin{proof}
	Anti-symmetry of $S_\nu$ is clear. To show that $[S_\nu,J] = 0$, notice that $\nu$ is a complex Lie bracket, so $\nu(J\f z,\cdot) = J \nu(\f z,\cdot) = 0$. Thus $J \f z \subset \f z$. Moreover since $J \in \f{so}(\f g, \bgm{\cdot}{\cdot})$, it follows that $J\f z^\perp \subset \f z^\perp$. Thus, $J$ commutes with the respective projection operators. The result now follows by direct computation using the fact that $J\nu = \frac{1}{2}\nu \id \wedge J$.
	\end{proof}
	We can now study the gauged bracket flow equation
	\begin{equation}
	\label{eqn:nilpGaugedBF}
	\dot \nu = -\pi(P_\nu - S_\nu)\nu,\qquad \nu(0) = \mu.
	\end{equation}
	Using this, we get 
	\begin{lem} 
	\label{lem:gaugedBFNilpSplit}	
	Let $\nu(t)$ be the solution to the gauged bracket flow (\ref{eqn:nilpGaugedBF}) with initial condition $\mu = \mu_0 + \mu_1$. Then $\f z (\nu(t)) = \f z(\mu) = \f z$ for all times and $\nu(t) = \nu_0(t) + \nu_1(t)$ where $\nu_0 := \Pr_{\f z} \nu$ and $\nu_1 := \Pr_{\f z^\perp}\nu$ solve
	\begin{equation}
	\label{eqn:gaugedBFNilpSplit}
	\begin{cases}
	\dot \nu_0 = \nu_0(\id \wedge \nu_1\nu_1^* - 2\nu_1^*\nu_1 - \nu_0^*\nu_0), \qquad & \nu_0(0) = \mu_0,\\
	\dot \nu_1 = -\pi(P_{\nu_1})\nu_1, \qquad & \nu_1(0) = \mu_1.
	\end{cases}
	\end{equation}
	\end{lem}
	\begin{proof}
	Suppose $(\nu_0, \nu_1)$ solves (\ref{eqn:gaugedBFNilpSplit}) and let $\nu := \nu_0 + \nu_1$. Then $\nu(0) = \mu$. Note that
	\[
	P_\nu - S_\nu = \nu_0\nu_0^* + 2 \nu_0\nu_1^* + \nu_1^*\nu_1.
	\]
	Thus,
	\[
	\begin{split}
	\dot \nu =& \dot\nu_0 + \dot \nu_1\\
	=& \nu_0(\id \wedge \nu_1 \nu_1^* - 2\nu_1^*\nu_1 - \nu_0^*\nu_0) - \pi(\nu_1\nu_1^*)\nu_1\\
	=& -\pi(\nu_1\nu_1^*)\nu - 2\nu_0\nu_1^*\nu_1 + 2\nu_1 \id \wedge \nu_0\nu_1^* -\nu_0\nu_0^*\nu_0 + \nu_0 \id \wedge \nu_0\nu_0^*\\
	=&-\pi(\nu_1\nu_1^*)\nu - \pi(2\nu_0\nu_1^*)\nu_1 -\pi(\nu_0\nu_0^*)\nu_0\\
	=& -\pi(\nu_1\nu_1^*)\nu - \pi(2\nu_0\nu_1^*)\nu -\pi(\nu_0\nu_0^*)\nu\\
	=& -\pi(P_\nu - S_\nu)\nu.
	\end{split}
	\]
	So $\nu$ solves (\ref{eqn:nilpGaugedBF}). By uniqueness of ODE solutions, the result follows.
	\end{proof}
	\begin{rmk}
	\label{rmk:inductiveFlow}
	$\nu_1$ is a Lie bracket with degree of nilpotency one less than that of $\mu$, and solves the ungauged bracket flow equation (\ref{eqn:bracketFlow}).
	\end{rmk}
	
	We can now analyse the system (\ref{eqn:gaugedBFNilpSplit}). The following two lemmas show long time existence and describe the asymptotic growth. This will be essential in investigating the limiting behaviour after a suitable normalisation. We start with an intermediate result for the gauged bracket flow.
	
	\begin{lem}
		\label{lem:gaugedBFAsymptotics}
		Let $\nu = \nu_0 + \nu_1 := \Pr_{\f z}\nu + \Pr_{\f z^\perp} \nu$ be a solution to the gauged bracket flow (\ref{eqn:gaugedBFNilpSplit}) with initial condition $\nu(0) = \nu$ where $\mu \in \vcla\setminus \{0\}$ is nilpotent. Suppose that $\mu_1$, considered as a solution to the bracket flow (\ref{eqn:bracketFlow}) exists for all time and there is a constant $C_{1} > 0$ such that
		\[
		\|\nu_1\|^2 < C_{1}t^{-1}
		\] 
		for all $t > 0$. Then the solution $\nu$ exists for all time and there is another constant $C > 0$ such that
		\[
		\|\nu\|^2 < Ct^{-1}.
		\]
		for all $t > 0$.
	\end{lem}
	\begin{proof}
		By Lemma \ref{lem:gaugedBFNilpSplit}, $(\nu_0,\nu_1)$ solves (\ref{eqn:gaugedBFNilpSplit}) on some interval $[0,T_{\max})$. Thus,
		\[
		\begin{split}
		\frac{d}{dt}\|\nu_0\|^2 &= 2 \tr(\dot\nu_0\nu_0^*)\\
		&=2 \tr (\nu_0(\id \wedge \nu_1 \nu_1^*)\nu_0^*) - 4\tr(\nu_0\nu_1^*\nu_1\nu_0^*) - 2\tr(\nu_0\nu_0^*\nu_0\nu_0^*)\\
		&\leq K \|\nu_0\|^2\|\nu_1\|^2 - L \|\nu_0\|^4,
		\end{split} 
		\]
		for some constants $K,L > 0$. So by hypothesis, 
		\[
		\frac{d}{dt}\|\nu_0\|^2 \leq \frac{M}{t} \|\nu_0\|^2 - L \|\nu_0\|^4.
		\]
		for some constant $M > 0$. Thus fixing a small $\varepsilon > 0$, $\|\nu_0\|^2 \leq f$ for $t \in [\varepsilon, T_{\max})$ where $f$ solves $\dot f = Mft^{-1} - Lf^2$ with $f(\varepsilon) = \|\nu_0(\varepsilon)\|^2 > 0$. Note that $(t^Mf^{-1})' = Lt^M$. Thus
		\[
		\begin{split}
		f(t)^{-1} &= \frac{LM}{M +1}\bigg(t - \varepsilon \Big(\frac{\varepsilon}{t}\Big)^M\bigg) + (\varepsilon t^{-1})^{M}f(\varepsilon)^{-1}\\
		&\geq \frac{LM}{M + 1}(t-\varepsilon).
		\end{split}
		\]
		Consequently,
		\[
		\|\nu_0\|^2 \leq C_2 (t-\varepsilon)^{-1},
		\]
		for all $t > \varepsilon$. Sending $\varepsilon \to 0$, we see that for $t > 0$, $\|\nu\|^2 = \|\nu_0\|^2 + \|\nu_1\|^2 \leq C t^{-1}$ on $[0,T_{\max})$. Long time existence follows from this fact and short time existence, as this means that $\nu(t)$ remains in a compact subset.
	\end{proof}

	With this, we can now describe precise the asymptotic growth of the bracket flow for nilpotent brackets.
	\begin{thm}
		\label{thm:BFNilpGaugedGrowth}
		Suppose $\mu(t)$ is a solution to the bracket flow (\ref{eqn:bracketFlow}) with initial condition $\mu(0) = \mu$ where $\mu \in \vcla\setminus\{0\}$ is nilpotent. Then, the solution is defined for all positive times, and there exists a constant $C > 0$ such that
		\[
		(Ct + \|\mu\|^{-2})^{-1}\leq \|\mu(t)\|^2 \leq Ct^{-1}.
		\]
		For all $t > 0$. In particular, $\mu(t) \to 0$ as $t \to \infty$.
	\end{thm}
	\begin{proof}
		We have that $\dot \mu(t) = -\pi(P_{\mu(t)})\mu(t) = -\pi(\mu(t)\mu(t)^*)\mu(t)$. Thus,
		\[
		\begin{split}
		\frac{d}{dt}\|\mu(t)\|^2 &= 2\bgm{\dot\mu(t)}{\mu(t)} \\
		&= -2\bgm{\pi(\mu(t)\mu(t)^*)\mu(t)}{\mu(t)}\\
		&\geq -2\|\pi(\mu(t)\mu(t)^*)\mu(t)\|\|\mu(t)\|\\
		&\geq -C \|\mu(t)\|^4,
		\end{split}
		\]
		where we have used the Cauchy-Schwartz inequality. Thus, we see that $\|\mu(t)\|^2 > \frac{1}{Ct + \|\mu(0)\|^{-2}}$ for all $t \geq 0$ by comparison with $\dot f = -Cf^2$.

		The proof of the second estimate is by induction. First, suppose $\mu$ is $2$-step nilpotent, so $\mu(\f g \wedge \f g) \subset \f z$ and the solution $\mu(t)$ to the bracket flow (\ref{eqn:bracketFlow}) satisfies
		\[
		\dot \mu(t) = -\mu(t)\mu(t)^*\mu(t).
		\]
		Thus, $\frac{d}{dt}\|\mu(t)\|^2 = - 2\|\mu(t)\mu(t)^*\|^2 \leq -c \|\mu(t)\|^4$, where $c > 0$ depends on dimension. Again, by comparison we have that $\|\mu(t)\|^2 < \frac{1}{ct}$ for all $t > 0$.
		
		Now suppose that for any $(k-1)$-step nilpotent bracket, the norm squared of the corresponding solution to the bracket flow (\ref{eqn:bracketFlow}) is bounded above by $C_1 t^{-1}$ for some constant $C_1 > 0$. If $\mu(t)$ solves (\ref{eqn:bracketFlow}), then by Theorem \ref{thm:gaugedBracketFlowEquiv} there is a smooth family $k_t \in \operatorname{U}(\f g, J)$ such that $k_t \cdot \mu(t)$ solves (\ref{eqn:nilpGaugedBF}). By Lemma \ref{lem:gaugedBFNilpSplit} and the inductive hypothesis, $\|\Pr_{\f z^\perp}k_t \cdot \mu(t)\|^2$ is bounded above by $C_1t^{-1}$. Thus, by Lemma \ref{lem:gaugedBFAsymptotics}, there exists $C > 0$ such that  $\|\mu(t)\|^2 = \|k_t \cdot \mu(t)\|^2 < C t^{-1}$ for all $t > 0$. Enlarging $C$ if necessary and combining with the first estimate we get that 
		\[
		\frac{1}{Ct + \|\mu(0)\|^{-2}} < \|\mu(t)\|^2 < \frac{C}{t},
		\]
		for all $t > 0$ as desired.
		\end{proof}

	We now turn to the analysis of the limiting behaviour of the normalised bracket flow
	\begin{equation}
	\label{eqn:BFNilpNorm}
	\frac{d}{dt}\tilde{\mu} = -\pi(P_{\tilde{\mu}} - \alpha_{\tilde{\mu}}\id)\tilde{\mu},\qquad \tilde{\mu}(0) = \frac{\mu}{\|\mu\|},
	\end{equation}
	where $\alpha_{\tilde{\mu}} := -\bgm{\pi(P_{\tilde{\mu}})\tilde{\mu}}{\tilde{\mu}}$ is chosen to keep $\|\tilde{\mu} \| \equiv 1$.

	In analysing (\ref{eqn:BFNilpNorm}), it will be crucial to consider a different (but asymptotically equivalent) normalisation of the gauged bracket flow (\ref{eqn:nilpGaugedBF}). Namely, the one that keeps $\|\Pr_{\f z^\perp}\nu\| \equiv 1$. The corresponding system is as follows
	\begin{equation}
		\label{eqn:BFNilpGaugedNorm}
		\begin{cases}
			\dot\eta_0 = \eta_0(\id \wedge \eta_1\eta_1^* - 2 \eta_1^*\eta_1 - \eta_0^*\eta_0 -  \alpha_{\eta_1}\id), &\eta_0(0) = \frac{\mu_0}{\|\mu_1\|},\\
			\dot \eta_1 = -\pi(P_{\eta_1} - \alpha_{\eta_1}\id)\eta_1, &\eta_1(0) = \frac{\mu_1}{\|\mu_1\|}.
		\end{cases}
	\end{equation}
	As before, $\alpha_{\eta_1} = -\bgm{\pi(P_{\eta_1})\eta_1}{\eta_1}$ and is now chosen to keep $\|\eta_1\| \equiv 1$. Notice that $\eta_1$ is a solution to the normalised bracket flow (\ref{eqn:BFNilpNorm}) with initial condition $\frac{\mu_1}{\|\mu_1\|}$. By induction, we will expect subconvergence of $\eta_1$ to a a soliton bracket. Thus, it is a natural first step to assume $\dot \eta_1 = 0$. In this setting, we can find an appropriate monotone quantity that is stationary precisely when $\dot \eta_0 = 0$. Specifically, define the map $\phi\colon \f g\otimes \Lambda^2 \f g^* \to \R_{\geq 0}$ by
	\begin{equation}
		\label{eqn:nilpMonotone}
	\phi(\eta) := \frac{1}{2}\big(\|\eta_0\eta_1^*\|^2 + \|\id \wedge \eta_1\eta_1^* - \eta_1^*\eta_1 - \eta_0^*\eta_0 - \alpha_{\eta_1}\id\|^2\big), \qquad \eta \in \f g\otimes \Lambda^2 \f g^*.
	\end{equation}
	We then have the following
	\begin{lem}
		\label{lem:monotoneQuantity}
		Suppose $(\eta_0,\eta_1)$ is a solution to (\ref{eqn:BFNilpGaugedNorm}) and that $\dot \eta_1 = 0$. Then $\phi(\eta_0 + \eta_1)$
		is monotone decreasing and stationary if and only if $\dot \eta_0 = 0$.
	\end{lem}
	\begin{proof}
	Define $H := \id \wedge \eta_1\eta_1^* - \eta_1^*\eta_1 - \alpha_{\eta_1}\id$ and $F := H-\eta_0^*\eta_0$. Then 
	\[
	\phi(\eta_0 + \eta_1) = \frac{1}{2}(\|\eta_0\eta_1^*\|^2 + \|F\|^2),
	\]
	and 
	\[
	\dot \eta_0 = \eta_0F - \eta_0\eta_1^*\eta_1 = \eta_0H - \eta_0\eta_1^*\eta_1 - \eta_0\eta_0^*\eta_0.
	\]
	Thus,
	\[
	\begin{split}
	\frac{d}{dt}\frac{1}{2}\|\eta_0\eta_1^*\|^2 &= \tr(\eta_0H\eta_1^*\eta_1\eta_0) - \|\eta_0^*\eta_0\eta_1^*\|^2 - \|\eta_1^*\eta_1\eta_0^*\|^2\\
	&= -\|\eta_0^*\eta_0\eta_1^*\|^2 - \|\eta_1^*\eta_1\eta_0^*\|^2,
	\end{split}
	\]
	since $0 = \dot \eta_1 = \eta_1 H$. Moreover, $\alpha_{\eta_1}$ only depends on $\eta_1$ and so it is constant. This gives
	\[
	\begin{split}
	\frac{d}{dt}\|F\|^2 &= -2\|\eta_0F\|^2 + 2\tr(\eta_1^*\eta_1\eta_0^*\eta_0F)\\
	&\leq -\|\eta_0F\|^2 + \|\eta_1^*\eta_1\eta_0^*\|^2, 
	\end{split}
	\]
	with equality if and only if $\eta_0F = \eta_0\eta_1^*\eta_1$. Here we have used the inequality $2\tr(AB^*)\leq \|A\|^2 + \|B\|^2$ with equality if and only if $A = B$, which follows from Cauchy-Schwarz and Young's inequalities. Thus,
	\[
	\frac{d}{dt}\phi \leq -\|\eta_0F\|^2 - \|\eta_0^*\eta_0\eta_1^*\|^2 \leq 0,
	\]
	with equality if and only if
	\[
	\eta_0F = \eta_0\eta_1^*\eta_1 = \eta_1\eta_0^*\eta_0 = 0,
	\]
	considered as vectors in $\f g \otimes \Lambda^2 \f g^*$. If this equality holds, then it is a simple computation to see that $\dot \eta_0 = 0$. Conversely if $\dot \eta_0 = 0$, then since $0 = \eta_1 H$,
	\[
	0 = \tr(\dot \eta_0\eta_1^*\eta_1\eta_0^*) = -\|\eta_0^*\eta_0\eta_1^*\|^2 - \|\eta_1^*\eta_1\eta_0^*\|^2.
	\]
	Thus, $\eta_0\eta_1^*\eta_1 = \eta_1\eta_0^*\eta_0 = 0$. Finally, $\eta_0F = \dot \eta_0 + \eta_0\eta_1^*\eta_1 = 0$ and so $\frac{d}{dt}\phi = 0$.
	\end{proof}

	\begin{cor}
		\label{cor:SubConvInductionStep}
		Suppose $(\eta_0,\eta_1)$ solves (\ref{eqn:BFNilpGaugedNorm}) and there exists an increasing sequence of times $t_l \to \infty$ such that $\eta_1(t_l)$ converges and $\dot \eta_1 (t_l)\to 0$ as $l \to \infty$. Then there exists another increasing sequence of times $t'_l \to \infty$ such that $(\eta_0(t_l'),\eta_1(t_l'))$ converges and $(\dot \eta_0(t'_l),\dot \eta_1(t'_l))\to 0$ as $l \to \infty$.
	\end{cor}
	\begin{proof}
		For any bracket $\nu = \nu_1 + \nu_0$ such that $\|\nu_1\| = 1$, let $\omega(\nu_0,\nu_1)$ denote the $\omega$-limit of (\ref{eqn:BFNilpGaugedNorm}) with initial condition replaced by $(\nu_0,\nu_1)$. 
		
		By hypothesis, $\eta_1(t_l) \to \bar \eta_1$ satisfying 	$\frac{d}{dt}\bar \eta_1 = 0$. By Theorem \ref{thm:BFNilpGaugedGrowth}, $\eta_0$ is bounded above and so after passing to a subsequence, $\eta_0(t_l)$ converges to a bracket $\bar \eta_0$. By Lemma \ref{lem:monotoneQuantity}, $\omega(\bar \eta_0,\bar\eta_1)$ consists entirely of fixed points to (\ref{eqn:BFNilpGaugedNorm}). Thus, $\omega(\bar \eta_0,\bar \eta_1) \subset \omega(\eta_0(0),\eta_1(0))$, and so there exists a fixed point to (\ref{eqn:BFNilpGaugedNorm}) in $\omega(\eta_0(0),\eta_1(0))$.
	\end{proof}
	\begin{rmk}
		By Proposition \ref{prop:algSol}, Corollary \ref{cor:SubConvInductionStep} implies that if $\eta_1$ sub-converges to an algebraic soliton, then $\eta = \eta_0 + \eta_1$ sub-converges to a semi-algebraic soliton. Which is in turn, algebraic by Theorem \ref{thm:nilpSolisAlg}.
	\end{rmk}
	This monotone quantity allows us to prove the following
	
	\begin{thm}
		\label{thm:NilpotentSubConvergence}
		Suppose $\tilde \mu(t)$ solves the normalised bracket flow equation (\ref{eqn:BFNilpNorm}). Then there exists an increasing sequence of times $t_l \to \infty$ such that $\tilde \mu(t_l)$ converges to a non-zero fixed point of (\ref{eqn:BFNilpNorm}) as $l \to \infty$.
	\end{thm}

	\begin{proof}
		The proof again is by induction on the degree of nilpotency of $\mu$. For two-step nilpotent brackets, the result follows from the proof of \cite[Theorem A]{pujia2020positive}.
		
		Now suppose that $\mu$ is $k$-step nilpotent and the result holds for all $(k-1)$-step nilpotent Lie brackets. Let $\eta(t) = \eta_0(t) + \eta_1(t)$, where $(\eta_0(t),\eta_1(t))$ solves (\ref{eqn:BFNilpGaugedNorm}). By induction there exists a sequence $t_l \to \infty$ such that $\eta_1(t_1)$ is convergent and $\frac{d}{dt}\eta_1(t_l) \to 0$. Next, by Corollary \ref{cor:SubConvInductionStep} there exists an increasing sequence $t_l' \to \infty$ such that $(\eta_0(t_l'),\eta_1(t_l')) \to (\bar \eta_0,\bar \eta_1)$ which is a fixed point of (\ref{eqn:BFNilpGaugedNorm}). Let $\tilde \mu(t)$ solve (\ref{eqn:BFNilpNorm}). By Theorem \ref{thm:gaugedBracketFlowEquiv}, there exists a family $k\colon \R_{\geq 0} \to \textrm{U}(\f g,J)$, a function $c\colon \R_{\geq 0} \to \R_{> 0}$ and a (bijective) time re-parametrisation $\tau\colon \R_{\geq 0} \to \R_{\geq0}$ such that
		\[
		\tilde \mu(t) = c(\tau(t))k(\tau(t))\cdot \eta(\tau(t)).
		\]
		By Theorem \ref{thm:BFNilpGaugedGrowth}, $c$ is bounded below and above. This along with compactness of $\textrm{U}(\f g,J)$ allows us to extract a convergent subsequence of $\tilde \mu(\tau^{-1}(t'_l))$ to a semi-algebraic soliton bracket. By Theorem \ref{thm:nilpSolisAlg}, this is an algebraic soliton and therefore a fixed point of (\ref{eqn:BFNilpNorm}) after normalisation.
	\end{proof}
	
We can now prove out main theorem.
	
	\begin{proof}[Proof of Theorem \ref{thm:nilpConvThm}]
	By Theorem \ref{thm:BracketFlowEquiv}, the $\HCFp$ is equivalent up to pull-back by biholomorphisms to the family of Hermitian manifolds corresponding to the brackets evolving under the bracket flow (\ref{eqn:bracketFlow}). Thus, it suffices to prove sub-convergence of this family to a soliton bracket after appropriate rescaling. Suppose $\mu_t$ solves (\ref{eqn:bracketFlow}). Then by Theorem \ref{thm:NilpotentSubConvergence}, there is an increasing sequence of times $t_l \to \infty$ such that $\mu_{t_l}/ \|\mu_{t_l}\|$ converges to a nilpotent, algebraic soliton bracket $\mu_\infty \neq 0$. By \cite[Corollary 6.20 (v)]{lauretConvHomMfds}, the corresponding simply-connected Lie groups converge in the Cheeger-Gromov sense to a non-flat simply-connected complex nilpotent Lie group. Theorem \ref{thm:nilpConvThm} then follows from Theorem \ref{thm:BFNilpGaugedGrowth}, and the fact that scaling a bracket by some $c \in \R$ is equivalent to scaling the metric on the corresponding Hermitian manifold by $c^{-1/2}$ (see \cite[\S 2.1]{lauretRicFlowHomMfd}).
	\end{proof}
	
	\section{$\HCFp$ on almost-abelian complex Lie groups}
	\label{sec:almostAbelian}
		\subsection{Almost-abelian complex  Lie groups}
		\label{subsec:almostAbelian}
	We say a complex Lie group $(G^{2n+2},J)$ is almost-abelian if its Lie algebra $\f{g}$ (viewed as a complex Lie algebra) admits a complex co-dimension one abelian ideal. Call this ideal $V$. We fix a background left-invariant Hermitian metric $\bgm{\cdot}{\cdot}$, and a left-invariant real $\bgm{\cdot}{\cdot}$-orthonormal frame $\{X_i\}_{i=0}^n\cup \{JX_i\}_{i=0}^n$, such that $\{X_i\}_{i=1}^{n}\cup\{JX_i\}_{i=1}^n$ spans $V$. There is then an orthogonal decomposition
	\begin{equation*}
		\label{eqn:almostAbelianDecomp}
		\f{g} = \R X_0 \oplus \R J X_0 \oplus V,
	\end{equation*}
	such that,
	\begin{equation*}
		\mu|_{V\times V} = 0, \qquad \mu(X_0,\f{g})\subset V.
	\end{equation*}
	One can easily check that $\{Z_i := \frac {1}{\sqrt{2}} (X_i - iJX_i)\}_{i=0}^{n} \subset \f g^{1,0}$ is a left-invariant, $\bgm{\cdot}{\cdot}$-unitary frame. Moreover, if we define $V^{1,0} := V\otimes \C \cap \f g^{1,0}$, there exists 
	\[
		A\in \f{gl}(V^{1,0}) \cong \f{gl}_n(\C),
	\]
	such that
	\begin{equation}
		\label{eqn:almostAbelianConst}
		\mu(Z_0,Z_i) = (0\oplus A)Z_i, \qquad \mu(Z_i,Z_j) = 0, \qquad \forall i,j \geq 1,
	\end{equation}
	where the direct sum is with respect to the decomposition 
	\[
		\f g^{1,0} = \C Z_0 \oplus V^{1,0}.
	\]
	Conversely, given $A \in \f{gl}_n(\C)$, we can let $\{Z_i\}_{i=0}^n$ be the standard basis of $\C^{n+1}$, and $\bgm{\cdot}{\cdot}$ the Hermitian inner product making $\{Z_i\}_{i=0}^n$ unitary. We define $\mu_A \in \f g \otimes \Lambda^2\f g ^*$ to be the bracket satisfying (\ref{eqn:almostAbelianConst}).
	
	From now, we will denote by $\mu_A$, the metric almost-abelian, complex Lie algebra arising from this construction. 
	
	\subsection{The behaviour of $\HCFp$ on almost-abelian Lie groups}
	In this section we investigate the asymptotic behaviour of the $\HCFp$ (\ref{eqn:HCF}) in the almost abelian setting and prove Theorem \ref{thm:aaConvThm}. We will do this using the bracket flow (\ref{eqn:bracketFlow}). We first compute the endomorphism driving the evolution, $P_{\mu_A}$.
	\begin{prop}
	\label{prop:almostAbelianP}
	For an almost abelian bracket $\mu_A$, $P_{\mu_A}$ is given by
	\[
	P_{\mu_A} = \begin{pmatrix}
	0&0\\
	0&AA^*
	\end{pmatrix},
	\]
	where $A^* \in \mathfrak{gl}_n(\C)$ is the $\bgm{\cdot}{\cdot}$-adjoint of $A$, and the blocks correspond to the subspace spanned by $Z_0$ and the abelian ideal respectively.
	\end{prop}
	\begin{proof}
	For $X \in \f g^{1,0} = \C Z_0 \oplus V^{1,0}$, by virtue of Lemma \ref{lem:PmuForm} we simply compute
	\[
	\begin{split}
	P_{\mu_A}X &= \sum_{i<j} g( X,\B{\mu_A(Z_i,Z_j)})\mu(Z_i,Z_j)\\
	&= (0\oplus A)\sum_j g( (0\oplus A)^*X,\B{Z_j}) Z_j\\
	&= (0\oplus A)(0\oplus A^*)X,
	\end{split}
	\]
	as required.
	\end{proof}
	We see that $P_{\mu_A}$ preserves the abelian ideal $V$, thus we obtain the following
	\begin{cor}
		\label{cor:muArel}
		Suppose $(\mu_t)_{t\in I}$ solves the bracket flow equation (\ref{eqn:bracketFlow}) on $0\in I \subset \R$ starting at $\mu_0 = \mu_A$. Then, $\mu_t$ remains almost abelian and is given by
		\[
			\mu_t(Z_0, \cdot) = 0\oplus A_t ,\qquad \mu_t(Z_i , Z_j) = 0, \qquad \forall i,j > 1,
		\]
		where $(A_t)_{t\in I}$ solves
		\begin{equation}
		\label{eqn:aabf}
				\dot A_t = A_t[A_t,A_t^*],\qquad A_0 = A.
		\end{equation}
	\end{cor}
	\begin{proof}
		Suppose $\mu_t$ satisfies (\ref{eqn:almostAbelianConst}) for the matrix $A_t$ evolving by (\ref{eqn:aabf}). Clearly $\mu_0 = \mu$. Further, for $i,j \geq 1$, $\dot \mu_t(Z_i,Z_j) = 0$ and,
		\[
		\begin{split}
		\dot \mu_t(Z_0,Z_i) &= (0 \oplus \dot A_t)Z_i\\
							&= (0 \oplus A_t[A_t,A_t^*])Z_i\\
							&=-(0\oplus A_tA_t^*)(0 \oplus A_t)Z_i + (0\oplus A_t)(0 \oplus A_tA_t^*)Z_i\\
							&= -P_{\mu_t}\mu_t(Z_0,Z_i) + \mu(Z_0,P_{\mu_t}Z_i) + \mu(P_{\mu_t}Z_0,Z_i)\\
							&= -(\pi(P_{\mu_t})\mu_t)(Z_0,Z_i).
		\end{split}
		\]
		Thus, $\mu_t$ satisfies (\ref{eqn:bracketFlow}) and uniqueness of ODE solutions implies the result.
	\end{proof}
	Our task is now to analyse the ODE (\ref{eqn:aabf}). We start with the following
	\begin{lem}
		\label{lem:specPres}
		Suppose $A_t$ solves (\ref{eqn:aabf}), then $\operatorname{spec}A_t = \operatorname{spec}A_0$.
	\end{lem}
	\begin{proof}
		Note that we have for $n \geq 1$
		\[
		\frac{\textrm{d}}{\textrm{d} t} \operatorname{tr}(A_t^n) = n \operatorname{tr}(A_t^{n-1}\dot {A}_t) = n\operatorname{tr}(A_t^{n-1}[A_t,A_tA_t^*]) = 0.
		\]
		Thus $\operatorname{tr}(A_t^n)$ is constant for all positive powers $n$. The conclusion follows from Newton's identities.
	\end{proof}
	Furthermore, we can observe that the norm of $A_t$ is a monotone quantity.
	\begin{prop}
		\label{prop:normMono}
		Suppose $A_t$ solves (\ref{eqn:aabf}). Then the quantity $\|A_t\|^2 = \operatorname{tr}(A_tA_t^*)$ is monotonically decreasing and stationary if and only if $A_t$ is normal.
	\end{prop}
	\begin{proof}
		Using (\ref{eqn:aabf}) and differentiating $\|A_t\|^2$ we get
		\[
		\begin{split}
		\frac{\textrm{d}}{\textrm{d}t} \textrm{tr}(A_tA_t^*) &= 2\textrm{tr}(\dot A_t A_t^*)\\
		&= 2 \textrm{tr}(A_t[A_t,A_t^*]A_t^*)\\
		&= 2\textrm{tr}(A_t^*A_tA_tA_t^*) - 2\|A_tA_t^*\|^2\\
		&\leq 2\|A_t^*A_t\|\|A_tA_t^*\| - 2\|A_tA_t^*\|^2\\
		&= 0.
		\end{split}
		\]
	Here we used the Cauchy-Swartz inequality. Equality holds if and only if $A_tA_t^* =cA_t^*A_t$ for some $c \in \C$. $A_t$ is clearly normal if it is $0$. Otherwise, taking the trace on both sides implies $c = 1$ and so $[A_t,A_t^*] = 0$.
	\end{proof}
	\begin{cor}
		\label{cor:AconvToNormMat}
		The solution $A_t$ exists for all $t\in [0,\infty)$. Moreover, the $\omega$-limit of $A_0$ under the flow (\ref{eqn:aabf}) consists of a single $U(n)$-orbit of normal operators.
	\end{cor}
	\begin{proof}
	Since $\|A_t\|^2$ is decreasing, $A_t$ remains in a compact subset and long-time existence follows. Suppose that for some sequence of times $t_k \to \infty$, $A_{t_k} \to A_\infty$. Then by Proposition \ref{prop:normMono}, $A_\infty$ is normal. Suppose that $A'_\infty$ is the limit of $A_t$ on another sequence of times. Then by Lemma \ref{lem:specPres}, $\operatorname{Spec}(A_\infty) = \operatorname{Spec}(A_0) = \operatorname{Spec}(A_\infty')$. Thus, since $A_\infty$ and $A_\infty'$ are normal, they must be unitarily similar. 
	\end{proof}

	We note here that long-time existence was proven by Ustinovskiy in \cite{UstHCFHom2017} for all solvable Lie algebras, however our proof gives precise information about the asymptotic behaviour. Before we prove Theorem \ref{thm:aaConvThm}, let us first observe one more fact about the long time behaviour of the flow.
	
	\begin{prop}
	\label{prop:convTo0}
	Let $\omega(A_0)$ denote $\omega$-limit of $A_0$ under (\ref{eqn:aabf}). Then $\omega(A_0) = \{0\}$ if and only if $A_0$ is nilpotent.	
	\end{prop}
	\begin{proof}
		Suppose $A_t$ solves (\ref{eqn:aabf}) and that for some sequence $t_k \to \infty$, $A_{t_k} \to A_\infty$. By Lemma \ref{lem:specPres} we have $\textrm{spec}A_\infty = \textrm{spec}A_0$. Thus $A_0$ is nilpotent if and only if $A_\infty$ is nilpotent. But by Corollary \ref{cor:AconvToNormMat}, $A_\infty$ is normal, so it is nilpotent if and only if it is $0$.
	\end{proof}

	We are now in a position to prove \ref{thm:aaConvThm}.
	\begin{proof}[Proof of Theorem \ref{thm:aaConvThm}]
	Let $(G,J,g)$ be an almost-abelian, non-nilpotent, complex Lie group with left-invariant Hermitian metric $g$ and suppose that $g_t$ solves (\ref{eqn:HCF}) with $g_0 = g$. Let $t_k \to \infty$ be an increasing sequence of times. If $\mu_t$ is the solution to the corresponding bracket flow (\ref{eqn:bracketFlow}), then by Proposition \ref{prop:normMono}, $\mu_{t_k}$ converges in the vector space topology after passing to a subsequence. This limit is non-zero by Proposition \ref{prop:convTo0}. By Corollary \ref{cor:AconvToNormMat}, the limit is a steady soliton bracket. Thus, by \cite[Corollary 6.20]{lauretConvHomMfds}, the corresponding limiting Lie group with left-invariant Hermitian metric $(G_\infty,J_\infty,g_\infty)$ is locally isometric to a steady soliton. We now claim that $(G_\infty,J_\infty,g_\infty)$ is simply-connected. Indeed, since $\mu_t$ remains almost-abelian, $G_{\mu_t}$ is diffeomorphic to $\mathbb{R}^{2(n+1)}$. Thus, by \cite[Theorem 8.3]{bohmLafImmRicFl}, $G_\infty$ is also diffeomorphic to $\mathbb{R}^{2(n+1)}$, and therefore simply-connected as claimed. It follows that $(G_\infty,J_\infty,g_\infty)$ is therefore globally isometric to a steady soliton.
	\end{proof}

	\section{Almost-abelian $\HCFp$ solitons}
	\label{sec:solitons}
	In this section we put together the necessary results to prove Theorem \ref{thm:solitonClass}. We first specialise to the case of solitons $\mu_A$ that are nilpotent. To do this, we will explicitly construct a certain canonical form for nilpotent almost abelian solitons that depends only on the Jordan form of $A$. This will prove both uniqueness and existence of almost-abelian nilpotent soliton metrics. To make this more precise, observe the following
	\begin{lem}
		\label{lem:aaOrbits}
		Suppose $A \in \f{gl}_n(\C)$ and denote by $\mu_A$ the corresponding almost-abelian complex Lie algebra. Then for any $A' \in \f{gl}_n(\C)$, $\mu_{A'} \in \operatorname{GL}_{n+1}(\C)\cdot\ \mu_A$ if and only if $A' = c h^{-1}Ah$ for some $c \in \C$ and $h \in \operatorname{GL}_n(\C)$. In this case, $\mu_{A'} \in \operatorname{U}(n+1)\cdot \mu_A$ if and only if $h \in \operatorname{U}(n)$ and $|c| = 1$.
		
	\end{lem} 
	
	\begin{proof}
		$\mu_{A'} \in \operatorname{GL}_{n+1}(\C) \cdot \mu_{A}$ if and only if there exists a Lie algebra isomorphism 
		\[
		H \colon (\C^{n+1},\mu_{A'}) \to (\C^{n+1},\mu_{A}).
		\]
		Since $H \in \operatorname{GL}_{n+1}(\C)$ preserves the abelian ideal, we assume that it has the form 
		\[
			H = \begin{pmatrix}c&0\\v&h\end{pmatrix} ,\qquad c\in \C\setminus \{0\},\qquad v\in \C^n,\qquad h\in\operatorname{GL}_{n}(\C),
		\]
		where the blocks correspond to the splitting $\C^{n+1} = \C Z_0 \oplus \C^n$. The first result follows from a simple computation using the fact that $H 0\oplus A' = H\mu_{A'}(Z_0,\cdot) = \mu_{A}(HZ_0,H\cdot)$. The second follows by observing that $H\in \operatorname{U}(n+1)$ if and only if $|c| = 1$, $v = 0$ and $h\in \operatorname{U}(n)$.
	\end{proof}
	
	After this, we will treat the non-nilpotent case and then show that these are the only such almost-abelian solitons. 
	
	Let us observe the following simple result.
	
	\begin{lem}
		\label{lem:aaSignOfLambda}
		Suppose $g$ is left-invariant $\HCFp$-soliton on an almost-abelian complex Lie group with bracket $\mu_A$. If $A$ is nilpotent then $g$ is an expanding soliton ($\lambda < 0$), otherwise, it is a steady soliton ($\lambda = 0$).
	\end{lem}
	\begin{proof}
		Let $g_t$ be the solution to \ref{eqn:HCF} with $g_0 = g$. Then $g_t$ is given by
		\[
		g(t) = (1-\lambda t)\phi_t^*g,
		\]
		for some family of biholomorphisms $\phi_t$. The solution exists for all times by Corollary \ref{cor:AconvToNormMat}, and so $\lambda \leq 0$. Now $\lambda < 0$ if and only if $1-\lambda t \to \infty$. Moreover in the varying brackets setting, scaling a metric by $c>0$ corresponds to scaling the associated bracket by $c^{-2}$ \cite[\S 2.1]{lauretRicFlowHomMfd}. Thus, $1-\lambda t \to \infty$ if and only if $A_t \to 0$, where $A_t$ solves (\ref{eqn:aabf}) with $A_0 = A$. This in turn is true if and only if $A$ is nilpotent by Proposition \ref{prop:convTo0}.
	\end{proof}
	\subsection{Classification of algebraic solitons}
	We now study the existence and uniqueness of left-invariant $\HCFp$ solitons on almost-abelian Lie groups. Let $A \in \f{gl}_n(\C)$. By Corollary \ref{cor:muArel} and Lemma \ref{lem:aaOrbits}, studying the existence of algebraic solitons on the Lie algebra $\mu_A$ amounts to looking at the existence (up to scaling) of matrices $B$ in the orbit $\textrm{GL}_{n}(\C)\cdot A$ satisfying
	\begin{equation}
	\label{eqn:aaalgSolEqn}
	B[B,B^*] = \lambda B,
	\end{equation}
	for some $\lambda \in \R$. Here the action of $\textrm{GL}_n(\C)$ on $\f{gl}_n(\C)$ is given by conjugation. That is, $h \cdot A = hAh^{-1}$ for $h \in \textrm{GL}_n(\C)$ and $A \in \f{gl}_n(\C)$.
	
	If $B$ solves (\ref{eqn:aaalgSolEqn}), then $B' = \alpha h\cdot B$ solves
	\[
	B'[B',B'] = \alpha^2 \lambda B',
	\]
	for any $\alpha \in \mathbb{R}$ and $h \in \textrm{U}(n)$. Thus, by Lemma \ref{lem:aaSignOfLambda} we may always assume $\lambda = 0$ when $B$ is non-nilpotent, and $\lambda = -1$ when $B$ is nilpotent. We will prove uniqueness of $B$ up to unitary transformation and scaling. By Lemma \ref{lem:aaOrbits}, this will be equivalent to uniqueness of algebraic solitons up to homotheties.
	
	Suppose $B\in \f {gl}_n(\C)$ solves, 
	\begin{equation}
	\label{eqn:aaNilpSolEqn}
	B[B,B^*] = -B, \qquad B^k \neq  0, \qquad B^{k+1} = 0,
	\end{equation}
	for $B \in \f{gl}_n(\C)$ and $k \geq 1$. That is, $B$ is $(k+1)-$step nilpotent and solves (\ref{eqn:aaalgSolEqn}) for $\lambda = -1$. 
	
	Since $B$ is nilpotent, it has a non-trivial kernel. We may decompose $\C^n = \ker B \oplus^\perp V_{\geq 1}$. Then, under this splitting,
	\[
	B = \begin{pmatrix}
	0&B_1\\
	0&B_2
	\end{pmatrix},
	\]
	where $B_1 = \Pr_{\ker B}B|_{V_{\geq 1}}$ and $B_2 = \Pr_{V_{\geq 1}}B|_{V_{\geq 1}} \in \f{gl}(V_{\geq 1},\C)$. By construction, $\ker B_1 \cap \ker B_2 = \{0\}$ and $B_2$ is $k$-step nilpotent. Equation (\ref{eqn:aaNilpSolEqn}) can now be observed under this splitting.
	\begin{lem}
		\label{lem:aaSolInitDecomp}
		$B = \begin{pmatrix}
		0&B_1\\
		0&B_2
		\end{pmatrix}$ solves (\ref{eqn:aaNilpSolEqn}) if and only if
		\[
		B_1B_2^* = 0 \qquad \text{and}\qquad B_1^*B_1 - [B_2,B_2^*] = \operatorname{Id}.
		\]
		In this case, $B_2[B_2,B_2^*] = -B_2$.
	\end{lem}
	\begin{proof}
		By a direct computation using (\ref{eqn:aaNilpSolEqn}) we get
		\[
		0 = B + B[B,B^*] = \begin{pmatrix}
		B_1B_2B_1^*& B_1([B_2,B_2^*]-B_1^*B_1 + \operatorname{Id})\\
		B_2B_2B_1^*& B_2([B_2,B_2^*]-B_1^*B_1 + \operatorname{Id})
		\end{pmatrix}.
		\]
		The result follows immediately from the fact that $\ker B_1$ and $\ker B_2$ intersect trivially.
	\end{proof}
	
	Define $V_1 := \ker B_{2}$ and $V_{\geq 2} := \ker B_1$. As a consequence of Lemma \ref{lem:aaSolInitDecomp} we obtain 
	\begin{lem}
		\label{lem:aaSolInitDecomp2}
		$V_{\geq 1} = V_1 \oplus^{\perp} V_{\geq 2}$.
	\end{lem}
	\begin{proof}
		Since $B_1B_2^* = 0$ we have
		\[
		V_1^\perp=(\ker B_2)^\perp = B_2^*V_{\geq 1} \subset \ker B_1 = V_{\geq 2}.
		\]
		Thus 
		\[
		\dim V_1 + \dim V_{\geq 2} = \dim V_{\geq 1} - \dim V_1^{\perp} + \dim V_2 \geq \dim V_{\geq 1}.
		\] 
		The result again follows from the trivial intersection of $V_1$ and $V_{\geq 2}$ and counting dimensions.
	\end{proof}
	Now, under the decomposition $\C^n = \ker B \oplus^\perp V_1 \oplus^\perp V_{\geq 2}$ we get
	\[
	B = \begin{pmatrix}
	0&E_1&0\\
	0&0&B'_1\\
	0&0&B'_2
	\end{pmatrix},
	\]
	where $E_1 := \Pr_{\ker B}B|_{V_1}$, $B'_1 := \Pr_{V_1} B|_{V_{\geq 2}}$, and $B'_2 := \Pr_{V_{\geq 2}}B|_{V_{\geq 2}}$. Note that $E_1$ is injective. Now
	\[
	B_2 = \begin{pmatrix}
			0&B'_1\\
			0&B'_2
		\end{pmatrix}
	\]
	solves (\ref{eqn:aaNilpSolEqn}) by Lemma \ref{lem:aaSolInitDecomp}. So we can proceed inductively applying the above Lemmas to obtain the following 
	
	\begin{prop}
	\label{prop:aaSolDecomp}
	Suppose $B \in \f{gl}_n(\C)$. Then $B$ solves (\ref{eqn:aaNilpSolEqn}) if and only if there exists an orthogonal decomposition
	\[
	\C^n = (V_0 := \ker B)\oplus^\perp V_1 \oplus^ \perp \dots \oplus^\perp V_k,
	\]
	so that
	\begin{equation}
	\label{eqn:BEiform}
	B = \begin{pmatrix}
	0&E_1&0&0&\dots &0\\
	0&0&E_2&0&\dots &0\\
	0&0&0&E_3&\dots&0\\
	\vdots&\vdots&\vdots&\vdots&\ddots&\vdots\\
	0&0&0&0&\dots&E_k\\
	0&0&0&0&\dots&0
	\end{pmatrix},
	\end{equation}
	where $E_i := \Pr_{V_{i-1}}B|_{V_i}$ solve the system
	\begin{equation}
	\label{eqn:aaSolEqnSystem}
	\begin{cases}
	E_i^*E_i - E_{i+1}E^*_{i+1} = \operatorname{Id}, \qquad &1\leq i < k, \\
	 E_k^*E_k = \operatorname{Id}, \qquad &i = k.
	\end{cases}
	\end{equation}
	Moreover, $\dim V_{i-1} - \dim V_{i}$ is the number  of $i\times i$ Jordan blocks appearing in the Jordan decomposition of $B$.
	\end{prop}
	\begin{proof}We start with the first direction. The decomposition and the form of $B$ follow from Lemmas \ref{lem:aaSolInitDecomp} and \ref{lem:aaSolInitDecomp2}  inductively. To see that (\ref{eqn:aaSolEqnSystem}) holds we simply use the fact that $B[B,B^*] = -B$ to get
	\[
	\begin{cases}
	E_i(E_i^*E_i - E_{i+1}E^*_{i+1}) = E_i, \qquad &1\leq i < k, \\
	E_kE_k^*E_k = E_k, \qquad &i = k.
	\end{cases}
	\]
	The maps $E_i$ are in injective by construction and so we may cancel on the left in the above relations to obtain (\ref{eqn:aaSolEqnSystem}).
	
	The converse follows by assuming $B$ is of the above form for operators $E_i$ solving (\ref{eqn:aaSolEqnSystem}) and verifying $B[B,B^*] = -B$.
	
	The only eigenvalue of $B$ is $0$. Thus, the number of Jordan blocks of size at least $i\times i$ is given by $\dim \ker B^{i} - \dim \ker B^{i-1} = \dim V_i$. So the number of Jordan blocks of size precisely $i\times i$ is given by $\dim V_{i-1} - \dim V_i$.
	\end{proof}
	\begin{prop}
	\label{prop:nilpSolCanForm}
	Suppose $B \in \f{gl}_n(\C)$ solves (\ref{eqn:aaNilpSolEqn}). 
	Let
	\[
	\C^n = V_0 \oplus V_1 \dots \oplus V_k,
	\]
	be the decomposition from Proposition \ref{prop:aaSolDecomp}.
	Then there exists an orthonormal basis of $\C^n$ subordinate to this splitting 
	in which $B$ is represented by the matrix
	\begin{equation}
	\label{eqn:nilpSolCanForm}
	B = \begin{pmatrix}
	0&\Sigma_1&0&0&\dots &0\\
	0&0&\Sigma_2&0&\dots &0\\
	0&0&0&\Sigma_3&\dots&0\\
	\vdots&\vdots&\vdots&\vdots&\ddots&\vdots\\
	0&0&0&0&\dots&\Sigma_k\\
	0&0&0&0&\dots&0
	\end{pmatrix},
	\end{equation}
	for matrices $\Sigma_i \in \C^{\dim V_{i-1} \times \dim V_i}, 1\leq i \leq k,$ whose only non-zero entries are positive real numbers on the main diagonal. Moreover, the entries $(\Sigma_i)_{jj} =: \sigma_{ij}, 1\leq i \leq k, 1\leq j \leq \dim{V_i}$ are given recursively by
	\begin{equation}
		\label{eqn:aaNilpSolEntries}
		\sigma_{ij} = \begin{cases}
		\sqrt{1 + (\sigma_{(i+1)j})^2}, \quad&i<k, 1\leq j \leq \dim V_{i+1},\\
		1, \quad &i < k, \dim V_{i+1} < j \leq \dim V_{i},\\
		1,\quad &i = k, 1\leq j \leq \dim V_k.
	\end{cases}
	\end{equation}
	In particular this matrix representation is uniquely determined by the numbers $\dim V_i$ and therefore depends only on the Jordan decomposition of $B$.
	\end{prop}
	
	\begin{proof}
	Suppose $B$ takes the form (\ref{eqn:BEiform}) from Proposition \ref{prop:aaSolDecomp} for operators $E_i \colon V_{i-1} \to V_i$ satisfying (\ref{eqn:aaSolEqnSystem}). Choosing an orthonormal basis for each space $V_i$, we consider the singular value decomposition of $E_i$ as a matrix.
	\[
	E_i = P_i\Sigma_i Q_i^*,
	\]
	where $P_i$ and $Q_i$ are unitary matrices whose columns consist of the eigenvectors of $E_iE_i^*$ and $E_i^*E_i$ respectively, and $\Sigma_i$ is a matrix whose only non-zero entries are positive real numbers on the main diagonal. By equation (\ref{eqn:aaSolEqnSystem}), $E_i^*E_i$ and $E_{i+1}E_{i+1}^*$ share the same eigenvectors, and so we may take $P_{i+1} = Q_i$. Then the matrix
	\[
	U := P_1 \oplus Q_1 \oplus Q_2 \oplus \dots \oplus Q_k,
	\]
	is unitary. Moreover, a simple computation using $P_{i+1} = Q_i$ shows that
	\[
	U^*BU = \begin{pmatrix}
	0&\Sigma_1&0&0&\dots &0\\
	0&0&\Sigma_2&0&\dots &0\\
	0&0&0&\Sigma_3&\dots&0\\
	\vdots&\vdots&\vdots&\vdots&\ddots&\vdots\\
	0&0&0&0&\dots&\Sigma_k\\
	0&0&0&0&\dots&0
	\end{pmatrix}.
	\]
	Thus, the first part of the proposition follows. Now by (\ref{eqn:aaSolEqnSystem}), the matrices $\Sigma_i \in \C^{\dim V_{i-1} \times \dim V_i}$ solve
	\[
	\begin{cases}
	\Sigma_i^*\Sigma_i - \Sigma_{i+1}\Sigma^*_{i+1} = \operatorname{Id}, \qquad &1\leq i < k, \\
	\Sigma_k^*\Sigma_k = \operatorname{Id}, \qquad &i = k.
	\end{cases}
	\]
	Equation (\ref{eqn:aaNilpSolEntries}) easily follows from the above relations.
	\end{proof}
	\begin{rmk}
		It follows that $\dim V_i \geq \dim V_{i+1}$ for all $0\leq i \leq k$ as each matrix $\Sigma_i$ is injective.
	\end{rmk}

	\begin{prop}
		\label{prop:aaNilpSolUniqueness}
		Suppose $B \in \f{gl}_n(\C)$ is $k+1$-step nilpotent and  solves (\ref{eqn:aaNilpSolEqn}). If $B' \in \operatorname{GL}_n(\C)\cdot B$ also solves (\ref{eqn:aaNilpSolEqn}), then $B' \in \operatorname{U}(n) \cdot B$.
	\end{prop}
	\begin{proof}
		If $B$ and $B'$ lie in the same $\textrm{GL}_n(\C)$ orbit and solve (\ref{eqn:aaNilpSolEqn}), then they have the same Jordan decomposition and thus Proposition \ref{prop:nilpSolCanForm} implies that we can chose two orthonormal bases in which $B$ and $B'$ take the form (\ref{eqn:nilpSolCanForm}). It follows that they must be unitarily similar.
	\end{proof}

	\begin{prop}
	\label{prop:nonNilpAASolIsAlg}
	Suppose $g$ is a left-invariant $\HCFp$-soliton on a non-nilpotent, almost abelian complex Lie group $G$ with bracket $\mu_A$. Then $A$ is normal with respect to $g$. In particular, $g$ is an algebraic soliton.
	\end{prop}
	\begin{proof}
		By Lemma \ref{lem:aaSignOfLambda}, $g$ must be a steady soliton. Thus, the solution $g_t$ to (\ref{eqn:HCF}) with $g_0 = g$ evolves only by pullback of biholomorphisms. Let $A_t$ solve (\ref{eqn:aabf}) with $A_0 = A$. By biholomorphism invariance of $\Theta$, we have that $\|A_t\|^2 = \operatorname{tr}(P_{\mu_{A_t}})=\operatorname{tr}_{g_t} \Theta_{\mu_{A}}(g_t)$ is constant. Proposition \ref{prop:normMono} implies $A$ is a normal matrix and therefore $\mu_A$ is an algebraic soliton bracket.
	\end{proof}
	We can now prove Theorem \ref{thm:solitonClass}
	\begin{proof}[Proof of Theorem \ref{thm:solitonClass}]
		We first note that all left-invariant solitons on a complex almost-abelian Lie group must be algebraic by Theorem \ref{thm:nilpSolisAlg} and Proposition \ref{prop:nonNilpAASolIsAlg}. Thus, it suffices to prove the theorem for algebraic solitons. Suppose $\mu_A$ is an almost-abelian bracket on $\C Z_0 \oplus \C^n$. 

		First, assume that $A$ is non-nilpotent. In this case, suppose $\mu_{A_1},\mu_{A_2} \in \operatorname{GL}_{n+1}(\C)\cdot \mu_A$ are both soliton brackets. Then by Proposition \ref{prop:nonNilpAASolIsAlg}, $A_1$ and $A_2$ are both normal. Thus, by Lemma \ref{lem:aaOrbits}, $A_1 = c k^*A_2k$ for some $c \in \C \setminus\{0\}$ and $k \in \operatorname{U}(n)$. Again by Lemma \ref{lem:aaOrbits}, $\mu_{\frac{1}{|c|}A_1} \in \operatorname{U}(n+1) \cdot \mu_{A_2}$. It follows that the solitons $(G_{\mu_{A_1}},J_{\mu_{A_1}},g_{\mu_{A_1}})$ and $(G_{\mu_{A_2}},J_{\mu_{A_2}},g_{\mu_{A_2}})$ must be equivariantly homothetic. The Lie group corresponding to the bracket $\mu_A$ admits a soliton if and only if there exists an almost-abelian bracket $\mu_{A'} \in \operatorname{GL}_{n+1}(\C) \cdot \mu_A$ such that $A'$ is normal (with respect to the background metric on $\C Z_0 \oplus \C^n)$. This in turn is true if and only if $\operatorname{ad}_{Z_0} = 0\oplus A$ is semi-simple by Lemma \ref{lem:aaOrbits}.

		Now assume $A$ is nilpotent. If $\mu_{A_1},\mu_{A_2} \in \operatorname{GL}_{n+1}\cdot \mu_A$ are soliton brackets, then $A_1,A_2$ are, up to scaling, similar to the same matrix by Proposition \ref{prop:nilpSolCanForm}. Uniqueness then follows from Proposition \ref{prop:aaNilpSolUniqueness}. For existence, it is clear from Proposition \ref{prop:nilpSolCanForm} that we can construct a matrix $A' \in \C^{n\times n}$ of the form (\ref{eqn:nilpSolCanForm}) with the same Jordan canonical form as $A$. Then $\mu_{A'}$ is a soliton bracket which lies in the orbit $\operatorname{GL}_{n+1}(\C) \cdot \mu_A$.
	\end{proof}

	\begin{rmk}
		\label{rmk:lattices}
		We note that almost-abelian, nilpotent Lie groups (which we have shown admit unique left-invariant $\HCFp$-solitons) all admit co-compact lattices. Indeed if $\mu_A$ denotes the corresponding Lie bracket, then the Jordan basis for $A$ has entries only in $\{0,1\}\subset \mathbb{Q}$. Thus, by Malcev's theorem \cite[Theorem 2.12]{RaghunDiscreteSubgroups}, it must have a co-compact lattice.
	\end{rmk}	
	\bibliography{Bibliography}
	\bibliographystyle{plain}
\end{document}